\documentclass{amsart}

\title[The pivotal cover and FS indicators]{The pivotal cover and Frobenius-Schur indicators}
\author{Kenichi Shimizu}

\usepackage{amsmath,amssymb,amsthm,amscd}
\usepackage{bbm}
\usepackage{graphicx}
\usepackage[all]{xy}

\usepackage{amsthm}
\numberwithin{equation}{section}

\newtheorem{counter}{}[section]

\theoremstyle{definition}
\newtheorem{definition}         [counter]{Definition}

\theoremstyle{plain}
\newtheorem{lemma}              [counter]{Lemma}
\newtheorem{proposition}        [counter]{Proposition}
\newtheorem{theorem}            [counter]{Theorem}
\newtheorem{corollary}          [counter]{Corollary}

\theoremstyle{remark}
\newtheorem{remark}             [counter]{Remark}

\newcommand{\id}{\mathrm{id}}
\newcommand{\ad}{\mathsf{ad}}
\newcommand{\ord}{\mathop{\rm ord}}
\newcommand{\Hom}{\mathop{\rm Hom}\nolimits}
\newcommand{\End}{\mathrm{End}}
\newcommand{\HOM}{\text{\sc Hom}}
\newcommand{\END}{\text{\sc End}}
\newcommand{\Trace}{\mathrm{Tr}}
\newcommand{\trace}{\mathrm{tr}}
\newcommand{\lptr}{\underline{\trace}^{(\ell)}}
\newcommand{\rptr}{\underline{\trace}^{(r)}}
\newcommand{\ldim}{\underline{\dim}^{(\ell)}}
\newcommand{\rdim}{\underline{\dim}^{(r)}}
\newcommand{\pivot}{\mathsf{p}}

\newcommand{\leftmod}[1]{{#1\text{\rm -}{\sf mod}}}

\newcommand{\YD}{\mathcal{YD}}

\newcommand{\assoc}{a}
\newcommand{\lunit}{\ell}
\newcommand{\runit}{r}
\newcommand{\unitobj}{\mathbbm{1}}
\newcommand{\eval}{{\rm ev}}
\newcommand{\coev}{{\rm coev}}

\newcommand{\op}{\mathsf{op}}
\newcommand{\piv}{\mathsf{piv}}
\newcommand{\rev}{\mathsf{rev}}
\newcommand{\str}{\mathsf{str}}

\newcommand{\rMon}{\mathsf{RigMon}}
\newcommand{\pMon}{\mathsf{PivMon}}
\begin{document}

\begin{abstract}
  In this paper, we introduce the notion of the pivotal cover $\mathcal{C}^{\piv}$ of a left rigid monoidal category $\mathcal{C}$ to develop a theoretical foundation for the theory of Frobenius-Schur (FS) indicators in ``non-pivotal'' settings. For an object $\mathbf{V} \in \mathcal{C}^{\piv}$, the $(n, r)$-th FS indicator $\nu_{n, r}(\mathbf{V})$ is defined by generalizing that of an object of a pivotal monoidal category. This notion gives a categorical viewpoint to some recent results on generalizations of FS indicators.

  Based on our framework, we also study the FS indicators of the ``adjoint object'' in a finite tensor category, which can be considered as a generalization of the adjoint representation of a Hopf algebra. The indicators of this object closely relate to the space of endomorphisms of the iterated tensor product functor.
\end{abstract}

\maketitle

\section{Introduction}
\label{sec:intro}

Linchenko and Montgomery \cite{MR1808131} introduced the $n$-th Frobenius-Schur (FS) indicator $\nu_n(V)$ of a finite-dimensional representation $V$ of a semisimple Hopf algebra by generalizing that for finite groups. This notion has been studied extensively and plays an important role in the study of semisimple Hopf algebras; see, {\it e.g.}, \cite{MR2213320} and references therein. It also has a categorical generalization as well as some other notions of Hopf algebra theory. Recall that a monoidal category $\mathcal{C}$ is said to be {\em left rigid} if every object $V \in \mathcal{C}$ has a left dual object $V^*$, and a left rigid monoidal category $\mathcal{C}$ is said to be {\em pivotal} if it is endowed with an isomorphism $\mathsf{p}: \id_{\mathcal{C}} \to (-)^{**}$ of monoidal functors. Based on the ``first formula'' \cite[\S2.3]{MR2213320} of the FS indicator for semisimple Hopf algebras, Ng and Schauenburg defined the $(n, r)$-th FS indicator $\nu_{n,r}(V) \in k$ of an object $V$ of a $k$-linear pivotal monoidal category $(\mathcal{C}, \otimes, \unitobj)$ to be the trace of the $r$-th power of a certain linear map
\begin{equation}
  \label{eq:intro-rotation}
  E_{V}^{(n)}: \Hom_{\mathcal{C}}(\unitobj, V^{\otimes n}) \to \Hom_{\mathcal{C}}(\unitobj, V^{\otimes n}).
\end{equation}
If $\mathcal{C}$ is the category of representations of a semisimple Hopf algebra, then $\mathcal{C}$ has a canonical pivotal structure and $\nu_{n,1}(V)$ agrees with Linchenko and Montgomery's. This generalization has some interesting applications to semisimple Hopf algebras, fusion categories and conformal field theories; see \cite{MR2313527,MR2725181,MR2774703,KenichiShimizu:2012}.

It is interesting to develop a theory of FS indicators for non-semisimple Hopf algebras (or, more generally, non-semisimple tensor categories). If a given monoidal category would be pivotal, then FS indicators could be defined in the way of Ng and Schauenburg. However, some interesting results have been obtained in ``non-pivotal'' settings: For example, Jedwab's trace invariant \cite{MR2724230} can be considered as a generalization of the second FS indicator to a non-pivotal setting. Kashina, Montgomery and Ng's $n$-th indicator \cite{KMN09} can be considered as a generalization of the $n$-th FS indicator of the regular representation to arbitrary finite-dimensional Hopf algebras.

For the above reasons, in this paper, we propose a generalization of FS indicators in ``non-pivotal'' settings. Our idea is very simple: Going back to its definition, we find that we only need the $V$-th component of the pivotal structure $\mathsf{p}: \id_{\mathcal{C}} \to (-)^{**}$ to define \eqref{eq:intro-rotation}. Now suppose that $\mathcal{C}$ is a linear left rigid monoidal category. Then $V \cong V^{**}$ does not hold in general. If we are given a {\em pair} $\mathbf{V} = (V, \phi)$ consisting of an object $V \in \mathcal{C}$ and an isomorphism $\phi: V \to V^{**}$, then we can define
\begin{equation}
  \label{eq:intro-rotation-2}
  E_{\mathbf{V}}^{(n)}: \Hom_{\mathcal{C}}(\unitobj, V^{\otimes n}) \to \Hom_{\mathcal{C}}(\unitobj, V^{\otimes n})
\end{equation}
by replacing $\mathsf{p}_V$ with $\phi$ in the definition of \eqref{eq:intro-rotation}. For $n \in \mathbb{Z}_{\ge 0} = \{ 0, 1, 2, \dotsc \}$ and $r \in \mathbb{Z}$, we define the $(n, r)$-th FS indicator of the pair $\mathbf{V}$ by
\begin{equation*}
  \nu_{n, r}(\mathbf{V}) = \Trace \Big( (E_{\mathbf{V}}^{(n)})^r \Big).
\end{equation*}
In the first half of this paper, we develop a category-theoretical framework to treat such a generalization of the FS indicators and explain how results of \cite{MR2724230} and \cite{KMN09} relate to our framework. In the second half, based on our framework, we study the indicators of the ``adjoint object'' in a finite tensor category, which can be considered as a generalization of the adjoint representation of a finite-dimensional Hopf algebra.

Now we describe the organization of this paper. In Section~\ref{sec:pre}, for reader's convenience, we remember some definitions from the theory of monoidal categories and Hopf monads. In Section~\ref{sec:duality-functor}, we first recall from \cite{MR2381536} the definition and the basic properties of the {\em left duality transformation}
\begin{equation}
  \label{eq:intro-dual-trans}
  \zeta^F_V: F(V^*) \xrightarrow{\quad \cong \quad} F(V)^*
\end{equation}
for a strong monoidal functor $F: \mathcal{C} \to \mathcal{D}$ between left rigid monoidal categories $\mathcal{C}$ and $\mathcal{D}$. We study its relation to adjunctions. In particular, we observe that, if $\mathcal{C}$ and $\mathcal{D}$ are rigid, and $I: \mathcal{D} \to \mathcal{C}$ is either left or right adjoint to $F$, then there is a natural isomorphism
\begin{equation}
  \label{eq:intro-dual-trans-adj}
  I(V^{**}) \xrightarrow{\quad \cong \quad} I(V)^{**}
\end{equation}
given by using~\eqref{eq:intro-dual-trans} (Lemma~\ref{lem:monoidal-adj-bidual}).

By a {\em pivotal object} in a left rigid monoidal category $\mathcal{C}$, we mean a pair $(V, \phi)$ consisting of an object $V$ and an isomorphism $\phi: V \to V^{**}$. In Section~\ref{sec:piv-cov}, we define the category $\mathcal{C}^{\piv}$ of pivotal objects in $\mathcal{C}$. This category has a natural structure of a pivotal monoidal category and a certain universal property. Thus we call it the {\em pivotal cover} of $\mathcal{C}$. By the universal property, a strong monoidal functor $F: \mathcal{C} \to \mathcal{D}$ between left rigid monoidal categories induces a strong monoidal functor
\begin{equation*}
  F^{\piv}: \mathcal{C}^{\piv} \to \mathcal{D}^{\piv}
\end{equation*}
preserving the pivotal structure. Moreover, if $\mathcal{C}$ and $\mathcal{D}$ are rigid and $F$ has a left (right) adjoint functor $I$, then $F^{\piv}$ has a left (right) adjoint functor given by
\begin{equation*}
  I^{\piv}: \mathcal{D}^{\piv} \to \mathcal{C}^{\piv},
  \quad (V, \phi) \mapsto \Big( I(V), \ I(V) \xrightarrow{\ I(\phi) \ }
  I(V^{**}) \xrightarrow[\ \eqref{eq:intro-dual-trans-adj} \ ]{\cong} I(V)^{**} \Big).
\end{equation*}

In Section~\ref{sec:FS-ind}, we define the $(n, r)$-th FS indicator $\nu_{n,r}(\mathbf{V})$ of a pivotal object $\mathbf{V}$ in a rigid monoidal category over a field $k$ in the way as described above. After proving its basic properties, we explain how Jedwab's trace invariant \cite{MR2724230} relate to our FS indicator (\S\ref{sec:example-trace-inv}). We also show that Kashina, Montgomery and Ng's $n$-th indicator $\nu_n^{\mathrm{KMN}}(H)$ of a finite-dimensional Hopf algebra over \cite{KMN09} can be expressed by using the above construction of adjunctions: We have
\begin{equation}
  \label{eq:intro-KMN}
  \nu_n^{\mathrm{KMN}}(H) = \nu_{n,1}(I^{\piv}(\mathbf{1})),
\end{equation}
where $I$ is a right adjoint of the fiber functor and $\mathbf{1}$ is the unit object of the pivotal cover of the category of finite-dimensional vector spaces over $k$ (\S\ref{subsec:FS-ind-examples}).

As these examples illustrate, interesting pivotal objects are obtained by adjunctions. Now let $\mathcal{C}$ be a finite tensor category \cite{MR2119143}. The main topic of Section~\ref{sec:ind-adj-obj} is the pivotal object $\mathbf{A}_{\mathcal{C}} = (A_{\mathcal{C}}, \phi_{\mathcal{C}})$ obtained as the image of the unit object under
\begin{equation*}
  \mathcal{C}^{\piv} \xrightarrow{\ I^{\piv} \ } \mathcal{Z}(\mathcal{C})^{\piv} \xrightarrow{\ U^{\piv} \ } \mathcal{C}^{\piv},
\end{equation*}
where $\mathcal{Z}(\mathcal{C})$ is the monoidal center of $\mathcal{C}$, $U: \mathcal{Z}(\mathcal{C}) \to \mathcal{C}$ is the forgetful functor, and $I$ is a right adjoint functor of $U$. As it generalizes the adjoint representation of a Hopf algebra (\S\ref{sec:adj-rep-Hopf}), we call $\mathbf{A}_{\mathcal{C}}$ (or its underlying object $A_{\mathcal{C}}$) the {\em adjoint object}. One of its important properties is that there exists an isomorphism
\begin{equation*}
  \Hom_{\mathcal{C}}(\unitobj, A_{\mathcal{C}}^{\otimes n}) \cong \END(\otimes^n),
\end{equation*}
where $\otimes^n: \mathcal{C}^n \to \mathcal{C}$ is the functor $(V_1, \dotsc, V_n) \mapsto V_1 \otimes \dotsb \otimes V_n$ and $\END(F)$ means the set of natural transformations from a functor $F$ to itself. Hence~\eqref{eq:intro-rotation-2} with $\mathbf{V} = \mathbf{A}_{\mathcal{C}}$ induces an operator $E_{\otimes}^{(n)}: \END(\otimes^n) \to \END(\otimes^n)$. We analyze this operator deeply by using the techniques of the graphical calculus and Hopf monads. In particular, we show that the $n$-th power of this operator is given by
\begin{equation*}
  (E_{\otimes}^{(n)})^n(\alpha)_{V_1, \dotsc, V_n} = (\alpha_{{}^{**}V_1, \dotsc, {}^{**}V_n})^{**}
  \quad (\alpha \in \END(\otimes^n), V_1, \dotsc, V_n \in \mathcal{C}).
\end{equation*}
This implies that the map $E_{\mathbf{A}_{\mathcal{C}}}^{(n)}$ is of finite order (Theorem~\ref{thm:En-order}) and thus $\nu_{n, r}(\mathbf{A}_{\mathcal{C}})$ is a cyclotomic integer. Note that our $\nu_{n,r}(\mathbf{V})$ is not a cyclotomic integer in general, while the FS indicator in the sense of Ng and Schauenburg \cite{MR2381536} always is.

In Section~\ref{sec:adj-rep-Hopf}, we apply our results to finite-dimensional Hopf algebras. We give a formula of $\nu_{n,r}(\mathbf{A}_{\leftmod{H}})$ for a finite-dimensional semisimple Hopf algebra $H$ by using the normalized integral in $H$. This formula cannot be generalized for non-semisimple $H$. Thus it is an interesting problem to know the value of $\nu_{n, r}(\mathbf{A}_{\leftmod{H}})$ for non-semisimple $H$. We determine $\nu_{1,r}(\mathbf{A}_{\leftmod{H}})$ for certain non-semisimple and non-pivotal Hopf algebras $H$ in a direct way and give some applications.

\section*{Acknowledgements}

The study of the ``adjoint object'' was greatly motivated by a talk given by Michihisa Wakui at a workshop held at Kyoto University in September 2012. The author thank him for valuable comments. The author also thank the referee for careful reading of the manuscript and pointing out a number of errors in the previous version. The author is supported by Grant-in-Aid for JSPS Fellows (24$\cdot$3606).  

\section{Preliminaries}
\label{sec:pre}

\subsection{Monoidal categories}

A {\em monoidal category} is a data $(\mathcal{C}, \otimes, \unitobj, \assoc, \lunit, \runit)$ consisting of a category $\mathcal{C}$, a bifunctor $\otimes: \mathcal{C} \times \mathcal{C} \to \mathcal{C}$ (called the tensor product), an object $\unitobj \in \mathcal{C}$ (called the unit object), and natural isomorphisms
\begin{gather*}
  \assoc_{X,Y,Z}: (X \otimes Y) \otimes Z \to X \otimes (Y \otimes Z) \quad (X, Y, Z \in \mathcal{C}), \\
  \lunit_X: \unitobj \otimes X \to X, \text{\quad and \quad} \runit_X: X \otimes \unitobj \to X \quad (X \in \mathcal{C})
\end{gather*}
satisfying the pentagon and the triangle axiom. We refer the reader to \cite{MR1797619,MR1321145,MR1712872} for the basic theory of monoidal categories. The natural isomorphisms $\assoc$, $\lunit$ and $\runit$ are called the associativity, the left unit, and the right unit isomorphism, respectively. Note that if $\mathcal{C} = (\mathcal{C}, \otimes, \unitobj, \assoc, \lunit, \runit)$ is a monoidal category, then 
\begin{equation*}
  \mathcal{C}^{\op} := (\mathcal{C}^{\op}, \otimes, \unitobj, \assoc^{-1}, \lunit^{-1}, \runit^{-1})
  \text{\quad and \quad}
  \mathcal{C}^{\rev} := (\mathcal{C}, \otimes^{\rev}, \unitobj, \assoc^{-1}, \runit, \lunit)
\end{equation*}
are also monoidal categories, where $X \otimes^{\rev} Y = Y \otimes X$ for $X, Y \in \mathcal{C}$.

Now let $\mathcal{C}$ and $\mathcal{D}$ be monoidal categories. A {\em monoidal functor} from $\mathcal{C}$ to $\mathcal{D}$ is a triple $F = (F, F_2, F_0)$ consisting of a functor $F: \mathcal{C} \to \mathcal{D}$, a natural transformation
\begin{equation*}
  F_2(X, Y): F(X) \otimes F(Y) \to F(X \otimes Y) \quad (X, Y \in \mathcal{C})
\end{equation*}
and a morphism $F_0: \unitobj \to F(\unitobj)$ in $\mathcal{D}$ satisfying certain axioms \cite[XI.2]{MR1712872}. If $F_2$ and $F_0$ are invertible, then $F$ is said to be {\em strong}.

Suppose that $F$ and $G$ are monoidal functors from $\mathcal{C}$ to $\mathcal{D}$. A {\em monoidal natural transformation} from $F$ to $G$ is a natural transformation $\alpha: F \to G$ such that
\begin{equation*}
  \alpha_{\unitobj} \circ F_0 = G_0 \text{\quad and \quad}
  \alpha_{X \otimes Y} \circ F_2(X, Y)
  = G_2(X, Y) \circ (\alpha_X \otimes \alpha_Y)
\end{equation*}
for all $X, Y \in \mathcal{C}$. As is well-known, monoidal categories, monoidal functors and monoidal natural transformations form a 2-category. An {\em equivalence of monoidal categories} (or {\em monoidal equivalence}) is defined as an equivalence in this 2-category.

Formally, a {\em comonoidal functor} from $\mathcal{C}$ to $\mathcal{D}$ is defined to be a monoidal functor from $\mathcal{C}^{\op}$ to $\mathcal{D}^{\op}$. In other words, it is a triple $F = (F, F^2, F^0)$ consisting of a functor $F: \mathcal{C} \to \mathcal{D}$, a natural transformation
\begin{equation*}
  F^2(X, Y): F(X \otimes Y) \to F(X) \otimes F(Y) \quad (X, Y \in \mathcal{C})
\end{equation*}
and a morphism $F^0: F(\unitobj) \to \unitobj$ in $\mathcal{D}$ satisfying the axioms similar to those for monoidal functors. A {\em comonoidal natural transformation} is defined in a similar way to the monoidal case.

\subsection{Rigidity}

Let $\mathcal{C}$ be a monoidal category. A {\em left dual object} of $V \in \mathcal{C}$ is a triple $(V^*, \varepsilon, \eta)$ consisting of an object $V^* \in \mathcal{C}$ and morphisms $\varepsilon: V^* \otimes V \to \unitobj$ (called the {\em evaluation}) and $\eta: \unitobj \to V \otimes V^*$ (called the {\em coevaluation}) satisfying
\begin{gather*}
  \runit_V \circ (\id_V \otimes \varepsilon)
  \circ \assoc_{V,V^*,V} \circ (\eta \otimes \id_V) = \lunit_V, \\
  \lunit_{V^*} \circ (\varepsilon \otimes \id_{V^*})
  \circ \assoc^{-1}_{V^*,V,V^*} \circ (\id_{V^*} \otimes \eta) = \runit_{V^*}.
\end{gather*}
Suppose that a left dual object $(V^*, \varepsilon, \eta)$ of $V$ is given. By using the evaluation and the coevaluation, one can establish isomorphisms
\begin{align*}
  \Hom_{\mathcal{C}}(X, V \otimes Y)
  & \cong \Hom_{\mathcal{C}}(V^* \otimes X, Y), \\
  \Hom_{\mathcal{C}}(X \otimes V, Y)
  & \cong \Hom_{\mathcal{C}}(X, Y \otimes V^*)
\end{align*}
natural in the variables $X, Y \in \mathcal{C}$ \cite[Chapter XIV]{MR1321145}. This means that the functors $V^* \otimes (-)$ and $(-) \otimes V$ are left adjoint functors of $V \otimes (-)$ and $(-) \otimes V^*$, respectively. The following lemma is essentially a consequence of the uniqueness of adjoint functors (see, {\it e.g.}, \cite[Lemma 2.1.5]{MR1797619}).

\begin{lemma}
  \label{lem:left-dual-uniq}
  Suppose that $(V_1^*, \varepsilon_1, \eta_1)$ and $(V_2^*, \varepsilon_2, \eta_2)$ are left dual objects of $V$. Then there exists an isomorphism $\gamma: V_1^* \to V_2^*$ characterized either of
  \begin{equation*}
    \varepsilon_1 = \varepsilon_2 \circ (\gamma \otimes \id_V)
    \text{\quad or \quad}
    \eta_2 = (\id_V \otimes \gamma) \circ \eta_1.
  \end{equation*}
\end{lemma}

We say that $\mathcal{C}$ is {\em left rigid} if every object of $\mathcal{C}$ has a left dual object. In this paper, {\em we always assume that a left dual object has been chosen for every object of a left rigid monoidal category}. The fixed left dual object of $V$ will be denoted by
\begin{equation*}
  (V^*, \ \eval_V: V^* \otimes V \to \unitobj, \ \coev_V: \unitobj \to V \otimes V^*).
\end{equation*}
A {\em right dual object} of $V \in \mathcal{C}$ is a left dual object of $V \in \mathcal{C}^{\rev}$. We say that $\mathcal{C}$ is {\em right rigid} if every object of $\mathcal{C}$ has a right dual object, or, equivalently, $\mathcal{C}^{\rev}$ is left rigid. Similarly to the above, we assume that a right dual object
\begin{equation*}
  ({}^*V, \ \eval'_V: V \otimes \, {}^*V \to \unitobj, \ \coev'_V: \unitobj \to {}^*V \otimes V)
\end{equation*}
has been chosen for each object $V$ of a right rigid monoidal category. We say that $\mathcal{C}$ is {\em rigid} if it is both left and right rigid.

Suppose that $\mathcal{C}$ is a left rigid monoidal category. Then $V \mapsto V^*$ extends to a strong monoidal functor $(-)^*: \mathcal{C}^{\op,\rev} \to \mathcal{C}$, called the {\em left duality functor} \cite[XIV.2]{MR1321145}. For a right rigid monoidal category, the {\em right duality functor} ${}^*(-)$ is defined in a similar way.

\subsection{Finite tensor categories}
\label{subsec:FTC}

Let $k$ be a field. By a {\em $k$-linear monoidal category}, we mean a monoidal category $\mathcal{C}$ such that
\begin{itemize}
\item $\Hom_{\mathcal{C}}(X, Y)$ is a finite-dimensional vector space over $k$ for all $X, Y \in \mathcal{C}$,
\item the composition of morphisms is $k$-bilinear, and
\item the tensor product of morphisms is $k$-bilinear.
\end{itemize}
Since the term ``tensor category'' has several different uses, we shall fix its meaning. In this paper, we follow the usage of Etingof and Ostrik \cite{MR2119143}. Thus a {\em tensor category} over $k$ is a $k$-linear abelian rigid monoidal category $\mathcal{C}$ such that
\begin{itemize}
\item every object of $\mathcal{C}$ has finite length, and
\item the unit object $\unitobj \in \mathcal{C}$ is simple.
\end{itemize}
One of our particular interests is a {\em finite tensor category} \cite{MR2119143}, {\it i.e.}, a tensor category which is, as a $k$-linear category, equivalent to the category of finite-dimensional modules over a finite-dimensional algebra over $k$.

\subsection{Hopf monad}

Let $\mathcal{C}$ be a category. A {\em monad} on $\mathcal{C}$ is a triple $(T, \mu, \eta)$ consisting of an endofunctor $T$ on $\mathcal{C}$ and natural transformations $\mu: T^2 \to T$ and $\eta: \id_{\mathcal{C}} \to T$ satisfying
\begin{equation*}
  \mu_V \circ T(\mu_V) = \mu_V \circ \mu_{T(V)}
  \text{\quad and \quad}
  \mu_V \circ \eta_{T(V)} = \id_{T(V)} = \mu \circ T(\eta_V)
\end{equation*}
for all $V \in \mathcal{C}$. Given a monad $T$ on $\mathcal{C}$, a {\em $T$-module} is a pair $(M, a_M)$ consisting of an object $M \in \mathcal{C}$ and a morphism $a_M: T(M) \to M$, called the {\em action}, satisfying
\begin{equation*}
  a_M \circ \mu_M = a_M \circ T(a_M)
  \text{\quad and \quad}
  a_M \circ \eta_{M} = \id_{M}.
\end{equation*}
Let $(M, a_M)$ and $(N, a_N)$ be two $T$-modules. A morphism $f: M \to N$ in $\mathcal{C}$ is said to be {\em $T$-linear} if $f \circ a_M = a_N \circ T(f)$. We denote by $\leftmod{T}$ the category of $T$-modules and $T$-linear morphisms between them.

\begin{remark}
  In literature ({\it e.g.} \cite{MR1712872}), a $T$-module is called a {\em $T$-algebra} and the category $\leftmod{T}$ is referred to as the {\em Eilenberg-Moore category} of $T$.
\end{remark}

Now suppose that $\mathcal{C}$ is a monoidal category. A {\em bimonad} on $\mathcal{C}$ is a monad $(T, \mu, \eta)$ on $\mathcal{C}$ such that the functor $T = (T, T^2, T^0)$ is comonoidal and the natural transformations $\mu$ and $\eta$ are comonoidal. If $T$ is a bimonad on $\mathcal{C}$, then $\leftmod{T}$ is a monoidal category: For $T$-modules $(M, a_M)$ and $(N, a_N)$, their tensor product is defined by $(M, a_M) \otimes (N, a_N) = (M \otimes N, a_{M \otimes N})$, where
\begin{equation*}
  \begin{CD}
    a_{M \otimes N}: T(M \otimes N) @>{T^2(M,N)}>> T(M) \otimes T(N) @>{a_M \otimes a_N}>> M \otimes N.
  \end{CD}
\end{equation*}
The pair $(\unitobj, T^0)$ is the unit object of $\leftmod{T}$. The associativity and unit constraints of $\leftmod{T}$ are defined so that the forgetful functor $\leftmod{T} \to \mathcal{C}$ is a strict monoidal functor.

A {\em Hopf monad} on $\mathcal{C}$ is a bimonad such that certain natural transformations, called the {\em fusion operators}, are invertible \cite{MR2793022}. If $\mathcal{C}$ is a rigid monoidal category, then the notions of a left antipode $S$ and a right antipode $\overline{S}$ for a bimonad $T$ on $\mathcal{C}$ are defined as natural transformations
\begin{equation*}
  S: T(T(-)^*) \to (-)^*
  \text{\quad and \quad}
  \overline{S}: T({}^*T(-)) \to {}^*(-)
\end{equation*}
satisfying certain axioms. A Hopf monad on $\mathcal{C}$ is characterized as a bimonad $T$ on $\mathcal{C}$ such that a left and a right antipode for $T$ exist \cite[\S3.4]{MR2793022}.

If $T$ is a Hopf monad on a rigid monoidal category $\mathcal{C}$, then the monoidal category $\leftmod{T}$ is rigid \cite[Theorem 3.8]{MR2355605}. More precisely, given a $T$-module $(M, a_M)$, we define $(M, a_M)^* = (M^*, a_{M^*})$, where
\begin{equation}
  \label{eq:Hopf-monad-dual-module}
  \begin{CD}
    a_{M^*}: T(M^*) @>{T(a_M^*)}>> T(T(M)^*) @>{S_M}>> M^*.
  \end{CD}
\end{equation}
The axioms for a left antipode ensure that the triple $((M, a_M)^*, \eval_M, \coev_M)$ is a left dual object of $(M, a_M)$. A right dual object of $M$ is given in a similar way but by using the right antipode $\overline{S}$.

\section{Duality functor}
\label{sec:duality-functor}

\subsection{Duality transformation}

Let $F: \mathcal{C} \to \mathcal{D}$ be a strong monoidal functor between left rigid monoidal categories $\mathcal{C}$ and $\mathcal{D}$. For each $V \in \mathcal{C}$, we set
\begin{equation*}
  \varepsilon_V^F = F_0^{-1} \circ F(\eval_V) \circ F_2(V^*, V)
  \text{\quad and \quad}
  \eta_V^F = F_2(V, V^*)^{-1} \circ F(\coev_V) \circ F_0.
\end{equation*}
Then the triple $(F(V^*), \varepsilon_V^F, \eta_V^F)$ is a left dual object of $F(V)$. Hence, by Lemma~\ref{lem:left-dual-uniq}, there exists a unique isomorphism $\zeta^F_V: F(V^*) \to F(V)^*$ characterized either of
\begin{equation*}
  \varepsilon_V^F = \eval_{F(V)} \circ (\zeta_V^F \otimes \id_{F(V)})
  \text{\quad or \quad}
  \eta_V^F = (\id_{F(V)} \otimes \zeta_V^F) \circ \coev_{F(V)}.
\end{equation*}
The family $\zeta^F = \{ \zeta_V^F:F(V^*) \to F(V)^* \}_{V \in \mathcal{C}}$ is in fact an isomorphism of strong monoidal functors \cite[\S1]{MR2381536}. Following \cite{MR2381536}, we call $\zeta^F$ the {\em duality transformation} for $F$. The aim of this section is to collect basic properties of the duality transformation and the natural isomorphism $\xi^F$ defined by
\begin{equation*}
  \begin{CD}
    \xi_V^F: F(V^{**}) @>{\zeta^F_{V^*}}>> F(V^*)^* @>{((\zeta_V^F)^{-1})^*}>> F(V)^{**}
    \quad (V \in \mathcal{C}).
  \end{CD}
\end{equation*}

\begin{lemma}
  \label{lem:dual-trans-composition}
  If $\displaystyle \mathcal{C} \mathop{\longrightarrow}^F \mathcal{D} \mathop{\longrightarrow}^G \mathcal{E}$ is a sequence of strong monoidal functors between left rigid monoidal categories, then, for all $V \in \mathcal{C}$, we have
  \begin{equation*}
    \zeta^{G \circ F}_V = \zeta_{F(V)}^G \circ G(\zeta_V^F)
    \text{\quad and \quad}
    \xi^{G \circ F}_V = \xi_{F(V)}^G \circ G(\xi_V^F).
  \end{equation*}
\end{lemma}
\begin{proof}
  The second equation follows immediately from the first. Hence we prove the first one. For $V \in \mathcal{C}$, we compute:
  \begin{align*}
    \eval_{G F(V)} \circ & ( (\zeta_{F(V)}^G \circ G(\zeta_V^F)) \otimes \id_{G F(V)} ) \\
    & = G_0^{-1} \circ G(\eval_{F(V)}) \circ G_2(F(V)^*, F(V)) \circ (G(\zeta_V^F) \otimes G(\id_{F(V)})) \\
    & = G_0^{-1} \circ G(\eval_{F(V)}) \circ G(\zeta_V^F \otimes \id_{F(V)}) \circ G_2(F(V^*), F(V)) \\
    & = G_0^{-1} \circ G (F_0^{-1} \circ F(\eval_V) \circ F_2(V^*, V) ) \circ G_2(F(V^*), F(V)) \\
    & = G_0^{-1} \circ G(F_0)^{-1} \circ G F(\eval_V) \circ G (F_2(V^*,V)) \circ G_2(F(V^*), F(V)) \\
    & = (G F)_0^{-1} \circ G F(\eval_V) \circ (G F)_2(V^*, V),
  \end{align*}
  where $(G F)_0$ and $(G F)_2$ are the monoidal structure of $G \circ F$. Now the first equation follows from the definition of the duality transformation.
\end{proof}

The following lemma is just a rephrasing of \cite[Proposition 7.1]{MR1250465} by using the duality transformation; see also \cite[Appendix A]{MR2095575}.

\begin{lemma}
  \label{lem:dual-trans-mon-nat-tr}
  Let $F, G: \mathcal{C} \to \mathcal{D}$ be strong monoidal functors between left rigid monoidal categories $\mathcal{C}$ and $\mathcal{D}$, and let $\alpha: F \to G$ be a monoidal natural transformation. Then $\alpha$ is invertible and we have
  \begin{equation*}
    \alpha_V^* \circ \zeta_V^G \circ \alpha_{V^*} = \zeta_V^F,
    \text{\quad and \quad} \alpha_{V^*} \circ (\zeta_V^F)^{-1} \circ \alpha_{V}^* = (\zeta_V^G)^{-1},
  \end{equation*}
  for all $V \in \mathcal{C}$. Hence,
  \begin{equation*}
    \alpha_V^{**} \circ \xi_V^G = \xi_V^F \circ \alpha_{V^{**}}.
  \end{equation*}
\end{lemma}

\begin{remark}
  For a strong monoidal functor $F: \mathcal{C} \to \mathcal{D}$ between right rigid monoidal categories, the {\em right duality transformation} $F({}^*V) \to {}^*F(V)$ is defined in a similar way. Since it is the left duality transformation of the strong monoidal functor $F^{\rev}: \mathcal{C}^{\rev} \to \mathcal{D}^{\rev}$ induced by $F$, the right duality transformations can be treated in a similar manner to the left duality transformations.
\end{remark}

\subsection{Pivotal structure}

A {\em pivotal monoidal category} is a left rigid monoidal category $\mathcal{C}$ endowed with a monoidal natural transformation $\pivot: \id_{\mathcal{C}} \to (-)^{**}$. Since the duality transformation for $(-)^{**}$ is the identity, $\pivot$ is invertible and
\begin{equation*}
  (\pivot_{V})^* = \pivot_{V^*}^{-1}
\end{equation*}
\cite[Appendix A]{MR2095575}. A pivotal monoidal category is rigid since the triple
\begin{equation*}
  (V^*, \ \eval_{V^*} \circ (\pivot_V \otimes \id_{V^*}), \ (\id_{V^*} \otimes \pivot_V^{-1}) \circ \coev_{V^*})
\end{equation*}
is a right dual object of $V$.

Let $F: \mathcal{C} \to \mathcal{D}$ be a strong monoidal functor between pivotal monoidal categories $\mathcal{C}$ and $\mathcal{D}$. We say that {\em $F$ preserves the pivotal structure} \cite{MR2381536} if
\begin{equation*}
  \pivot_{F(V)} = \xi^F_{V} \circ F(\pivot_V)
\end{equation*}
for all $V \in \mathcal{C}$, where $\pivot$'s are the pivotal structures of $\mathcal{C}$ and $\mathcal{D}$. The following result follows from Lemma~\ref{lem:dual-trans-composition}.

\begin{lemma}
  \label{lem:piv-preserve-func-compose}
  Let $\displaystyle \mathcal{C} \mathop{\longrightarrow}^F \mathcal{D} \mathop{\longrightarrow}^G \mathcal{E}$ be a sequence of strong monoidal functors between pivotal monoidal categories. If $F$ and $G$ preserve the pivotal structure, then so does their composition $G \circ F$.
\end{lemma}

\subsection{Duality and adjunctions}
\label{subsec:dual-adj}

Following Mac Lane \cite[IV.1]{MR1712872}, we write
\begin{equation*}
  \langle F, G, \eta, \varepsilon \rangle: \mathcal{C} \rightharpoonup \mathcal{D}
\end{equation*}
if $F: \mathcal{C} \to \mathcal{D}$ is a functor, $G: \mathcal{D} \to \mathcal{C}$ is right adjoint to $F$, and $\eta: \id_{\mathcal{C}} \to G F$ and $\varepsilon: F G \to \id_{\mathcal{D}}$ are the unit and the counit of the adjunction. The following lemma is well-known (see, {\it e.g.}, \cite[Proposition 3.84]{MR2724388}):

\begin{lemma}
  \label{lem:monoidal-adj-1}
  Let $  \langle F, G, \eta, \varepsilon \rangle: \mathcal{C} \rightharpoonup \mathcal{D}$ be an adjunction between monoidal categories $\mathcal{C}$ and $\mathcal{D}$. If $G$ is monoidal, then $F$ is comonoidal by
  \begin{equation*}
    F^2(X, Y) = \varepsilon_{F(X) \otimes F(Y)} \circ F \Big( G_2(F(X), F(Y)) \circ (\eta_{X} \otimes \eta_{Y}) \Big)
    \quad (X, Y \in \mathcal{C})
  \end{equation*}
  and $F^0 = \varepsilon_{\unitobj} \circ F(G_0)$. Similarly, if $F$ is comonoidal, then $G$ is monoidal.
\end{lemma}

By using the duality transformation, we prove:

\begin{lemma}
  \label{lem:monoidal-adj-2}
  Let $F: \mathcal{D} \to \mathcal{C}$ be a strong monoidal functor between rigid monoidal categories. Suppose that $F$ has a left or a right adjoint functor and let $I$ be one of them. Then there are isomorphisms
  \begin{gather*}
    I(V \otimes F(M)) \cong I(V) \otimes M
    \text{\quad and \quad}
    I(F(M) \otimes V) \cong M \otimes I(V)
  \end{gather*}
  natural in the variables $M \in \mathcal{D}$ and $V \in \mathcal{C}$.
\end{lemma}
\begin{proof}
  We consider the case where $I$ is left adjoint to $F$. For all $M, N \in \mathcal{D}$ and $V \in \mathcal{C}$, there are the following natural isomorphisms:
  \begin{align*}
    \Hom_{\mathcal{D}}(I(V \otimes F(M)), N)
    & \cong \Hom_{\mathcal{C}}(V \otimes F(M), F(N))
    & & (\text{adjunction}) \\
    & \cong \Hom_{\mathcal{C}}(V, F(N) \otimes F(M)^*)
    & & (\text{left duality}) \\
    & \cong \Hom_{\mathcal{C}}(V, F(N \otimes M^*))
    & & (\text{use $F_2$ and $\zeta^F$}) \\
    & \cong \Hom_{\mathcal{D}}(I(V), N \otimes M^*)
    & & (\text{adjunction}) \\
    & \cong \Hom_{\mathcal{D}}(I(V) \otimes M, N)
    & & (\text{left duality}).
  \end{align*}
  Hence $I(V \otimes F(M)) \cong I(V) \otimes M$ by Yoneda's lemma. The other cases are proved in a similar way.
\end{proof}

\begin{remark}
  The same proof applies in the case where $\mathcal{C}$ and $\mathcal{D}$ are closed ({\it i.e.}, admits an internal Hom) and $F$ preserves internal Hom's.
\end{remark}

\begin{remark}
  Let $F: \mathcal{C} \to \mathcal{D}$ be as in Lemma~\ref{lem:monoidal-adj-2}. By emphasizing the use of Hopf monadic techniques, natural isomorphisms in that lemma are obtained as follows: We first consider the case where $I$ is left adjoint to $F$. Let $\varepsilon: I F \to \id_{\mathcal{C}}$ be the counit of the adjunction. By Lemma~\ref{lem:monoidal-adj-1}, $I$ is comonoidal. Moreover, the pair $(I, F)$ is a Hopf adjunction by \cite[Proposition~3.5]{MR2793022}. Hence the left fusion operator $\kappa$ for this Hopf adjunction, defined by
  \begin{equation}
    \label{eq:l-Hopf-operator}
    \begin{CD}
      \kappa_{V,M}: I(V \otimes F(M))
      @>{I^2}>> I(V) \otimes I F(M)
      @>{\id \otimes \varepsilon}>> I(V) \otimes M
    \end{CD}
  \end{equation}
  for $V \in \mathcal{C}$ and $M \in \mathcal{D}$, is a natural isomorphism. Similarly, the right fusion operator gives a natural isomorphism $I(F(M) \otimes V) \cong M \otimes I(V)$. If $I$ is right adjoint to $F$, then the functor $I^{\op}: \mathcal{D}^{\op} \to \mathcal{C}^{\op}$ induced by $I$ is left adjoint to $F^{\op}: \mathcal{C}^{\op} \to \mathcal{D}^{\op}$. Applying the above arguments to the pair $(I^{\op}, F^{\op})$, we obtain desired natural isomorphisms.
\end{remark}

For a functor $J$ between rigid monoidal categories, we set $J^!(X) = {}^*\!J(X^*)$.

\begin{lemma}[{\cite[Lemma 3.5]{MR2869176}}]
  \label{lem:monoidal-adj-dual}
  Let $F: \mathcal{C} \to \mathcal{D}$ be a strong monoidal functor between rigid monoidal categories. If $F$ has a left adjoint $I_{\ell}: \mathcal{D} \to \mathcal{C}$, then $I_{\ell}^!$ is right adjoint to $F$. Similarly, if $F$ has a right adjoint $I_r$, then $I_r^!$ is left adjoint to $F$.
\end{lemma}

In other words, if one of a left adjoint $I_{\ell}$ or a right adjoint $I_r$ of $F$ exists, then the other one exists and there are natural isomorphisms
\begin{equation*}
  I_{\ell}(V^*) \cong I_r(V)^* \text{\quad and \quad} I_r(V^*) \cong I_{\ell}(V)^* \quad (V \in \mathcal{C}).
\end{equation*}
Hence we obtain natural isomorphisms
\begin{equation*}
  I_{\ell}(V^{**}) \cong I_r(V^*)^* \cong I_{\ell}(V)^{**}
  \text{\quad and \quad}
  I_{r}(V^{**}) \cong I_{\ell}(V^*)^* \cong I_{r}(V)^{**}.
\end{equation*}
One can obtain explicit descriptions of the above isomorphisms by going back to the proof of \cite[Lemma 3.5]{MR2869176}. However, instead of doing so, we prove the following assertion directly:

\begin{lemma}
  \label{lem:monoidal-adj-bidual}
  Let $F: \mathcal{C} \to \mathcal{D}$ be a strong monoidal functor between rigid monoidal categories. If $\langle I_{\ell}, F, \eta, \varepsilon \rangle: \mathcal{D} \rightharpoonup \mathcal{C}$ is an adjunction, then the composition
  \begin{equation*}
    \begin{CD}
      \xi^{(\ell)}_V: I_{\ell}(V^{**}) @>{I_{\ell}(\eta^{**})}>> I_{\ell}(F I_{\ell}(V)^{**})
      @>{I_{\ell}(\xi^{-1})}>> I_{\ell} F(I_{\ell}(V)^{**}) @>{\varepsilon}>> I_{\ell}(V)^{**}
    \end{CD}
  \end{equation*}
  is an isomorphism. Similarly, if $\langle F, I_{r}, \overline{\eta}, \overline{\varepsilon} \, \rangle: \mathcal{C} \rightharpoonup \mathcal{D}$ is an adjunction, then
  \begin{equation*}
    \begin{CD}
      \xi^{(r)}_V: I_{r}(V)^{**} @>{\overline{\eta}}>> I_{r} F(I_{r}(V)^{**})
      @>{I_{r}(\xi)}>> I_{r} (F I_{r}(V)^{**}) @>{I_{r}(\overline{\varepsilon}^{**})}>> I_{r}(V^{**})
    \end{CD}
  \end{equation*}
  is an isomorphism.
\end{lemma}
\begin{proof}
  There is the following sequence of natural isomorphisms:
  \begin{align*}
    \Hom_{\mathcal{D}}(I_{\ell}(V)^{**}, X^{**})
    & \cong \Hom_{\mathcal{D}}(I_{\ell}(V), X)
    & & (f^{**} \leftrightarrow f) \\
    & \cong \Hom_{\mathcal{D}}(V, F(X))
    & & (\text{adjunction}) \\
    & \cong \Hom_{\mathcal{D}}(V^{**}, F(X)^{**})
    & & (f \leftrightarrow f^{**}) \\
    & \cong \Hom_{\mathcal{D}}(V^{**}, F(X^{**}))
    & & (\text{use $\xi_X: F(X^{**}) \to F(X)^{**}$}) \\
    & \cong \Hom_{\mathcal{D}}(I_{\ell}(V^{**}), X^{**})
    & & (\text{adjunction}).
  \end{align*}
  Let $\Phi: \Hom_{\mathcal{D}}(I_{\ell}(V)^{**}, I_{\ell}(V)^{**}) \to \Hom_{\mathcal{D}}(I_{\ell}(V^{**}), I_{\ell}(V)^{**})$ be the composition of the above isomorphisms with $X = I_{\ell}(V)$. Then $\xi^{(\ell)}_V = \Phi(\id)$. Note that the functor $(-)^{**}$ is an equivalence, since we have assumed $\mathcal{C}$ to be rigid. Thus, by Yoneda's lemma, $\xi_V^{(\ell)}$ is invertible. The invertibility of $\xi_V^{(r)}$ follows from the same argument applied to $\langle I_r^{\op}, F^{\op}, \overline{\varepsilon}^{\op}, \overline{\eta}^{\op} \rangle: \mathcal{D}^{\op} \rightharpoonup \mathcal{C}^{\op}$.
\end{proof}

Suppose moreover that $\mathcal{C}$ and $\mathcal{D}$ are pivotal monoidal categories and $F$ preserves the pivotal structure in Lemma~\ref{lem:monoidal-adj-bidual}. The following lemma says that adjoint functors of such an $F$, in a sense, preserves the pivotal structure. By abuse of notation, we denote the pivotal structures of $\mathcal{C}$ and $\mathcal{D}$ by the same symbol $\pivot$. Then:

\begin{lemma}
  \label{lem:piv-preserve-func-adj}
  For all $V \in \mathcal{C}$, we have
  \begin{equation*}
    \xi^{(\ell)}_{V} \circ I_{\ell}(\pivot_V) = \pivot_{I_{\ell}(V)}
    \text{\quad and \quad}
    I_{r}(\pivot_V) = \xi^{(r)}_V \circ \pivot_{I_{r}(V)}.
  \end{equation*}
\end{lemma}
\begin{proof}
  Consider the following diagram:
  \begin{equation*}
    \xymatrix@C+=48pt{
      I_{\ell}(V^{**})
      \ar[r]^{I_{\ell}(\eta^{**})}
      & I_{\ell}(F I_{\ell}(V)^{**})
      \ar[r]^{I_{\ell}(\xi^{-1})}
      & I_{\ell} F( I_{\ell}(V)^{**})
      \ar[r]^{\varepsilon}
      & I_{\ell}(V)^{**} \\
      I_{\ell}(V)
      \ar[r]^{I_{\ell}(\eta)}
      \ar[u]^{I_{\ell}(\pivot)}
      & I_{\ell} F I_{\ell}(V) \ar[r]^{\id}
      \ar[u]^{I_{\ell}(\pivot)}
      & I_{\ell} F I_{\ell}(V)
      \ar[r]^{\varepsilon}
      \ar[u]_{I_{\ell} F(\pivot)}
      & I_{\ell}(V)
      \ar[u]_{\pivot}
    }
  \end{equation*}
  The left and the right square commute by the naturality of $\pivot$ and $\varepsilon$, respectively, and so does the middle square since $F$ preserves the pivotal structure. The composition along the top and the bottom row are the morphism $\xi_V^{(\ell)}$ and the identity, respectively. Therefore we get the first equation. The second one is obtained in a similar way.
\end{proof}

\section{Pivotal cover}
\label{sec:piv-cov}

\subsection{Pivotal cover}
\label{subsec:pivotal-cover}

For a strong monoidal endofunctor $F: \mathcal{C} \to \mathcal{C}$ on a monoidal category $\mathcal{C}$, we define the category $\mathcal{C}^F$ of ``objects fixed by $F$'' as follows: An object of $\mathcal{C}^F$ is a pair $(V, \varphi)$ consisting of an object $V \in \mathcal{C}$ and an isomorphism $\varphi: V \to F(V)$ in $\mathcal{C}$. The set of morphisms is defined by
\begin{equation*}
  \Hom_{\mathcal{C}^F}((V, \varphi), (W, \psi))
  = \{ f \in \Hom_{\mathcal{C}}(V, W) \mid \psi \circ f = F(f) \circ \varphi \}
\end{equation*}
for $(V, \varphi), (W, \psi) \in \mathcal{C}^F$, and the composition of morphisms is defined in an obvious way. Note that $\mathcal{C}^F$ is a monoidal category with tensor product
\begin{equation*}
  (V, \varphi) \otimes (W, \psi) = (V \otimes W, \ F_2(V, W) \circ (\varphi \otimes \psi))
\end{equation*}
and unit $(\unitobj, F_0)$. If, moreover, $\mathcal{C}$ is left rigid, then so is $\mathcal{C}^F$. Indeed,
\begin{equation*}
  ((V, \varphi)^*, \eval_V, \coev_V),
  \text{\quad where \quad}
  (V, \varphi)^* = (V^*, (\varphi^* \circ \zeta_V^F)^{-1}),
\end{equation*}
is a left dual object of $(V, \varphi) \in \mathcal{C}^F$ (The fact that $\eval_V$ and $\coev_V$ are morphisms in $\mathcal{C}^F$ follows from the definition of the duality transformation $\zeta^F$).

Now we apply the above construction to the double-dual functor $(-)^{**}$:

\begin{definition}
  For a left rigid monoidal category $\mathcal{C}$, we denote by $\mathcal{C}^{\piv}$ the left rigid monoidal category $\mathcal{C}^F$ with $F = (-)^{**}$ and call it the {\em pivotal cover} of $\mathcal{C}$.
\end{definition}

An object of $\mathcal{C}^{\piv}$ will be referred to as a {\em pivotal object}. In this paper, a pivotal object is usually denoted in bold letters. Unless otherwise noted, given a pivotal object, say $\mathbf{V}$, its underlying object is denoted by the corresponding letter $V$ and the isomorphism $V \to V^{**}$ is denoted by $\phi_V$.

To justify our terminology, we first prove the following theorem:

\begin{theorem}
  $\mathcal{C}^{\piv}$ is a pivotal monoidal category with pivotal structure
  \begin{equation*}
    \pivot_{\mathbf{V}} = \phi_V \quad (\mathbf{V} \in \mathcal{C}^{\piv}).
  \end{equation*}
\end{theorem}
\begin{proof}
  Let $\mathbf{V} \in \mathcal{C}^{\piv}$. Since the duality transformation for $(-)^{**}$ is the identity \cite[Appendix A]{MR2095575}, the left dual object of $\mathbf{V}$ is given by
  \begin{equation*}
    (\mathbf{V}^*, \eval_V, \coev_V)
    \text{\quad where \quad}
    \mathbf{V}^* = (V^*, (\phi_V^{*})^{-1}).
  \end{equation*}
  In particular, $\mathbf{V}^{**} = (V^{**}, \phi_V^{**})$. From this, one easily sees that $\pivot_{\mathbf{V}}: \mathbf{V} \to \mathbf{V}^{**}$ is indeed an isomorphism in $\mathcal{C}^{\piv}$. For a morphism $f: \mathbf{V} \to \mathbf{W}$ in $\mathcal{C}^{\piv}$, we have
  \begin{equation*}
    \pivot_{\mathbf{V}} \circ f = \phi_V \circ f = f^{**} \circ \phi_W = f^{**} \circ \pivot_{\mathbf{W}},
  \end{equation*}
  {\it i.e.}, $\pivot = \{ \pivot_{\mathbf{V}} \}_{\mathbf{V} \in \mathcal{C}^{\piv}}$ is a natural isomorphism. That $\pivot$ is a monoidal natural transformation follows directly from the definition of the tensor product of $\mathcal{C}^{\piv}$.
\end{proof}

We define the forgetful functor $\Pi_{\mathcal{C}}: \mathcal{C}^{\piv} \to \mathcal{C}$ by $\mathbf{V} \mapsto V$ for $\mathbf{V} \in \mathcal{C}^{\piv}$. This functor is faithful and strict monoidal. The pivotal cover has the following universal property: 

\begin{theorem}
  \label{thm:piv-cov-univ}
  Let $\mathcal{C}$ be a pivotal monoidal category, and let $\mathcal{D}$ be a left rigid monoidal category. For every strong monoidal functor $F: \mathcal{C} \to \mathcal{D}$, there uniquely exists a strong monoidal functor $\widetilde{F}: \mathcal{C} \to \mathcal{D}^{\piv}$ preserving the pivotal structure such that
  \begin{equation}
    \label{eq:piv-cov-univ-1}
    \Pi_{\mathcal{D}} \circ \widetilde{F} = F
  \end{equation}
  as monoidal functors
\end{theorem}
\begin{proof}
  Define a functor $\widetilde{F}: \mathcal{C} \to \mathcal{D}^{\piv}$ by
  \begin{equation}
    \label{eq:piv-cov-univ-2}
    \widetilde{F}(V) = (F(V), \ \xi_V^F \circ F(\pivot_V))
    \text{\quad and \quad}
    \widetilde{F}(f) = F(f)
  \end{equation}
  for all objects $V \in \mathcal{C}$ and all morphisms $f$ in $\mathcal{C}$, and endow this functor with the monoidal structure given by
  \begin{equation}
    \label{eq:piv-cov-univ-3}
    \widetilde{F}_2(V, W) = F_2(V, W) \quad (V, W \in \mathcal{C})
    \text{\quad and \quad}
    \widetilde{F}_0 = F_0.
  \end{equation}
  One can check that $\widetilde{F}$ so defined satisfies the required conditions. Hence the existence is proved.

  Now we show the uniqueness. Let $\widetilde{F}$ be a functor satisfying the required conditions. By \eqref{eq:piv-cov-univ-1}, we have $\widetilde{F}(f) = F(f)$ for all morphisms $f$ in $\mathcal{C}$, and $\widetilde{F}(V) = (F(V), \varphi_V)$ for some $\varphi_V: F(V) \to F(V)^{**}$ if $V \in \mathcal{C}$ is an object. Since the duality transformation of $\Pi_{\mathcal{D}}$ is the identity, we have
  \begin{equation*}
    \zeta^{\widetilde{F}}_V = \zeta^F_V
    \text{\quad and \quad}
    \xi^{\widetilde{F}}_V = \xi^F_V
    \quad \text{(as morphisms in $\mathcal{D}$)}
  \end{equation*}
 for all $V \in \mathcal{C}$. Since $\widetilde{F}$ preserves the pivotal structure,
  \begin{equation*}
    \varphi_V = \pivot_{\widetilde{F}(V)} = \xi_V^{\widetilde{F}} \circ \widetilde{F}(\pivot_V) = \xi_V^F \circ F(\pivot_V),
  \end{equation*}
  where $\pivot$'s are the pivotal structures of $\mathcal{C}$ and $\mathcal{D}^{\piv}$. Hence $\widetilde{F}$ is given by~\eqref{eq:piv-cov-univ-2}. It is obvious from \eqref{eq:piv-cov-univ-1} that the monoidal structure of $\widetilde{F}$ is given by~\eqref{eq:piv-cov-univ-3}.
\end{proof}

\begin{remark}
  \label{rem:piv-cov-pdim}
  For an endomorphism $f: V \to V$ in a pivotal monoidal category with pivotal structure $\pivot$, the left and the right pivotal trace of $f$ are the endomorphisms of the unit object defined and denoted by
  \begin{equation*}
    \lptr(f) = \eval_V (\id \otimes \pivot_V^{-1} f) \coev_{V^*}
    \text{\quad and \quad}
    \rptr(f) = \eval_{V^*} (\pivot_V f \otimes \id) \coev_{V},
  \end{equation*}
  respectively. We call $\ldim(V) := \lptr(\id_V)$ and $\rdim(V) := \rptr(\id_V)$ the left and the right pivotal dimension of $V$, respectively. By the definition of the pivotal structure of $\mathcal{C}^{\piv}$, we have
  \begin{equation*}
    \ldim(\mathbf{V}) = \eval_V (\id \otimes \phi_V^{-1}) \coev_{V^*}
    \text{\quad and \quad}
    \rdim(\mathbf{V}) = \eval_{V^*} (\phi_V \otimes \id) \coev_{V}
  \end{equation*}
  for all $\mathbf{V} \in \mathcal{C}^{\piv}$.
\end{remark}

\begin{remark}
  Given a Hopf algebra $H$ with bijective antipode $S$, we define the Hopf algebra $H^{\piv}$ to be the semidirect product $H^{\piv} = H \rtimes \langle t \rangle$, where $\langle t \rangle$ is (the group algebra of) the infinite cyclic group with generator $t$ acting on $H$ as $S^2$. If $V$ is a finite-dimensional left $H^{\piv}$-module, then
  \begin{equation*}
    \langle \phi_V(v), f \rangle = \langle f, t v \rangle
    \quad (v \in V, f \in V^*)
  \end{equation*}
  defines an $H$-linear isomorphism $\phi_V: V \to V^{**}$. The assignment $V \mapsto (V, \phi_V)$ gives an isomorphism $\leftmod{H^{\piv}} \cong (\leftmod{H})^{\piv}$ of monoidal categories.

  Observe that $H^{\piv}$ is infinite-dimensional even if $H$ is finite-dimensional. Thus one might think that the pivotal cover is often ``too big'' and some ``smaller'' category should be considered. For example, a {\em pivotalization} \cite{MR2183279} of a fusion category $\mathcal{C}$, which is effectively used in their study of fusion categories, is a full subcategory of the pivotal cover $\mathcal{C}^{\piv}$. As such an example illustrates, some smaller subcategories of $\mathcal{C}^{\piv}$ are often more convenient than $\mathcal{C}^{\piv}$ itself. However, in this paper, we consider the whole $\mathcal{C}^{\piv}$ since its universal property will play an important role in what follows.
\end{remark}

\subsection{Functoriality of the pivotal cover}
\label{subsec:piv-cov-functorial}

Now we investigate functorial properties of our construction. First, suppose that $F: \mathcal{C} \to \mathcal{D}$ is a strong monoidal functor between left rigid monoidal categories. Applying Theorem~\ref{thm:piv-cov-univ} to $F \circ \Pi_{\mathcal{C}}$, we get a unique strong monoidal functor
\begin{equation*}
  F^{\piv}: \mathcal{C}^{\piv} \to \mathcal{D}^{\piv}
\end{equation*}
such that $F^{\piv}$ preserves the pivotal structure and $\Pi_{\mathcal{D}} \circ F^{\piv} = F \circ \Pi_{\mathcal{C}}$ as monoidal functors. By the proof of Theorem~\ref{thm:piv-cov-univ}, $F^{\piv}$ is given explicitly by
\begin{equation*}
  F^{\piv}(\mathbf{V}) = (F(V), \ \xi^F_V \circ F(\phi_V))
\end{equation*}
for $\mathbf{V} \in \mathcal{C}^{\piv}$. If, in addition, $G: \mathcal{D} \to \mathcal{E}$ is another strong monoidal functor between left rigid monoidal categories, then we have
\begin{equation*}
  (G \circ F)^{\piv} = G^{\piv} \circ F^{\piv}
\end{equation*}
by Theorem~\ref{thm:piv-cov-univ} (or, more directly, by Lemma~\ref{lem:dual-trans-composition}).

Fix left rigid monoidal categories $\mathcal{C}$ and $\mathcal{D}$. For a monoidal natural transformation $\alpha: F \to G$
between strong monoidal functors $F, G: \mathcal{C} \to \mathcal{D}$, we set
\begin{equation*}
  \alpha^{\piv}: F^{\piv} \to G^{\piv};
  \quad \alpha^{\piv}_{\mathbf{V}} = \alpha_V
  \quad (\mathbf{V} \in \mathcal{C}^{\piv}).
\end{equation*}
By Lemma~\ref{lem:dual-trans-mon-nat-tr}, we see that $\alpha^{\piv}$ is in fact a monoidal natural transformation. It is obvious that the assignment $\alpha \mapsto \alpha^{\piv}$ preserves the vertical composition and the horizontal composition of monoidal natural transformations.

Now we introduce two 2-categories $\rMon_{\ell}$ and $\pMon$ as follows:
\begin{itemize}
\item The 0-cells, 1-cells and 2-cells of $\rMon_{\ell}$ are left rigid monoidal categories, strong monoidal functors, and monoidal natural transformations, respectively.
\item The 0-cells, 1-cells and 2-cells of $\pMon$ are pivotal monoidal categories, pivotal-structure-preserving strong monoidal functors, and monoidal natural transformations, respectively (see Lemma \ref{lem:piv-preserve-func-compose}).
\end{itemize}
In both cases, the composition of cells is defined in the same way as the 2-category of monoidal categories. Summarizing our results so far, we obtain:

\begin{theorem}
  $(-)^{\piv}: \rMon_{\ell} \to \pMon$ is a 2-functor.
\end{theorem}

The following corollary, which is conceptually obvious, is rigorously proved by using the functoriality of the pivotal cover.

\begin{corollary}
  If $F: \mathcal{C} \to \mathcal{D}$ is a monoidal equivalence between left rigid monoidal categories, then $F^{\piv}: \mathcal{C}^{\piv} \to \mathcal{D}^{\piv}$ is a monoidal equivalence preserving the pivotal structure.
\end{corollary}

\begin{remark}
  \label{rem:piv-cov-2-adj}
  Let $\mathcal{F}: \pMon \to \rMon_{\ell}$ be the 2-functor forgetting the pivotal structure. The pivotal cover $\mathcal{P} := (-)^{\piv}$ is a left adjoint of $\mathcal{F}$ in the following sense: For a pivotal monoidal category $\mathcal{C}$, we define
  \begin{equation*}
    \Sigma_{\mathcal{C}}: \mathcal{C} \to \mathcal{P}\mathcal{F}(\mathcal{C}),
    \quad V \mapsto (V, \pivot_V),
  \end{equation*}
  where $\pivot$ is the pivotal structure of $\mathcal{C}$. $\Sigma_{\mathcal{C}}$ is a strong monoidal functor and preserves the pivotal structure. Then, for all 0-cells $\mathcal{C} \in \pMon$ and $\mathcal{D} \in \rMon_{\ell}$, we have
  \begin{equation*}
    \Pi_{\mathcal{F}(\mathcal{C})} \circ \mathcal{F}(\Sigma_{\mathcal{C}}) = \id_{\mathcal{F}(\mathcal{C})}
    \text{\quad and \quad}
    \mathcal{P}(\Pi_{\mathcal{D}}) \circ \Sigma_{\mathcal{P}(\mathcal{D})} = \id_{\mathcal{P}(\mathcal{D})},
  \end{equation*}
  where the forgetful functor $\Pi_{\mathcal{X}}: \mathcal{X}^{\piv} \to \mathcal{X}$ for $\mathcal{X} \in \rMon_{\ell}$ is regarded as a strong monoidal functor $\Pi_{\mathcal{X}}: \mathcal{F}\mathcal{P}(\mathcal{X}) \to \mathcal{X}$.
\end{remark}

\begin{remark}
  Define the 2-category $\rMon_r$ of right rigid monoidal categories in the same way as $\rMon_{\ell}$. The assignment $\mathcal{C} \mapsto \mathcal{C}^{\rev}$ extends to isomorphisms
  \begin{equation*}
    (-)^{\rev}: \rMon_r \to \rMon_{\ell}
    \text{\quad and \quad}
    (-)^{\rev}: \pMon \to \pMon
  \end{equation*}
  of 2-categories. One can define the pivotal cover $\mathcal{P}_r: \rMon_r \to \pMon$  for right rigid monoidal categories by
  \begin{equation*}
    \begin{CD}
      \mathcal{P}_r: \rMon_r @>{(-)^{\rev}}>> \rMon_{\ell} @>{\mathcal{P}}>> \pMon @>{(-)^{\rev}}>> \pMon.
    \end{CD}
  \end{equation*}
\end{remark}

\subsection{Adjunctions and the pivotal cover}

Here we show that an adjoint of a strong monoidal functor $F$ induces an adjoint of $F^{\piv}$. In the following theorem, we use the same notations as in Lemma~\ref{lem:monoidal-adj-bidual}.

\begin{theorem}
  \label{thm:piv-cov-induced-adj}
  Let $F: \mathcal{C} \to \mathcal{D}$ be a strong monoidal functor between rigid monoidal categories. If $\langle I_{\ell}, F, \varepsilon, \eta \rangle: \mathcal{C} \rightharpoonup \mathcal{D}$ is an adjunction, then the functor
  \begin{equation*}
    I_{\ell}^{\piv}(\mathbf{V}) := (I_{\ell}(V),
    \, \xi^{(\ell)}_{V} \circ I_{\ell}(\phi_V))
    \quad (\mathbf{V} \in \mathcal{C}^{\piv})
  \end{equation*}
  is left adjoint to $F^{\piv}$. Similarly, if $\langle F, I_{r}, \overline{\eta}, \overline{\varepsilon} \, \rangle: \mathcal{C} \rightharpoonup \mathcal{D}$ is an adjunction, then
  \begin{equation*}
    I_{r}^{\piv}(\mathbf{V}) := (I_{r}(V),
    \, (\xi^{(r)}_{V})^{-1} \circ I_r(\phi_V))
    \quad (\mathbf{V} \in \mathcal{C}^{\piv})
  \end{equation*}
  is right adjoint to $F^{\piv}$.
\end{theorem}
\begin{proof}
  We only show that $I_{\ell}^{\piv}$ is left adjoint to $F^{\piv}$. Define
  \begin{equation*}
    \eta^{\piv}_{\mathbf{X}} = \eta_X
    \text{\quad and \quad}
    \varepsilon^{\piv}_{\mathbf{X}} = \varepsilon_X
  \end{equation*}
  for $\mathbf{X} \in \mathcal{D}^{\piv}$ and $\mathbf{V} \in \mathcal{C}^{\piv}$, respectively. Then $\eta_{\mathbf{X}}^{\piv}: \mathbf{X} \to F^{\piv}_{}I^{\piv}_{\ell}(\mathbf{X})$ is a morphism in $\mathcal{D}^{\piv}$. Indeed, writing $F^{\piv}_{}I^{\piv}_{\ell}(\mathbf{X}) = (F I(X), \phi_{F I(X)})$, we compute
  \begin{align*}
    \phi_{F I(X)} \circ \eta_X
    & = \xi_{I(X)} \circ F(\varepsilon_{I(X)^{**}})
    \circ F I (\xi_{I(X)}^{-1} \circ \eta_X^{**} \circ \phi_X) \circ \eta_X \\
    & = \xi_{I(X)} \circ F(\varepsilon_{I(X)^{**}}) \circ \eta_{F(I(V)^{**})}
    \circ \xi_{I(X)}^{-1} \circ \eta_X^{**} \circ \phi_X \\
    & = \eta_X^{**} \circ \phi_X.
  \end{align*}
  It is obvious from the definition that $\eta^{\piv}: \id_{\mathcal{D}^{\piv}} \to F^{\piv} \circ I_{\ell}^{\piv}$ is a natural transformation. One can show that $\varepsilon^{\piv}: I_{\ell}^{\piv} \circ F^{\piv} \to \id_{\mathcal{C}^{\piv}}$ is a natural transformation in a similar way. Hence we obtain an adjunction $\langle I^{\piv}_{\ell}, F^{\piv}, \eta^{\piv}, \varepsilon^{\piv} \rangle$.
\end{proof}

\begin{remark}
  \label{rem:F-piv-Frobenius}
  We use the same notations as in the above theorem. One can apply results in \S\ref{subsec:dual-adj} to $I_{\ell}^{\piv}$ and $I_r^{\piv}$. For example, we have natural isomorphisms
  \begin{equation}
    \label{eq:piv-cov-induced-adj-dual}
    I_{\ell}^{\piv}(\mathbf{V}^*) \cong I_{r}^{\piv}(\mathbf{V})^*
    \text{\quad and \quad}
    I_{r}^{\piv}(\mathbf{V}^*) \cong I_{\ell}^{\piv}(\mathbf{V})^*
  \end{equation}
  by Lemma~\ref{lem:monoidal-adj-dual}. In particular, $I_{\ell}^{\piv}(\mathbf{1}) \cong I_{r}^{\piv}(\mathbf{1})^*$, where $\mathbf{1}$ is the unit object. Hence, if $I_{\ell}^{\piv}$ and $I_{r}^{\piv}$ would be isomorphic as functors, then $I_{\ell}^{\piv}(\mathbf{1})$ would be self-dual.

  A functor $F$ is said to be {\em Frobenius} \cite{MR1926102} if it has a left adjoint functor which is right adjoint to $F$. Based on the above argument, we will give an example of a strong monoidal functor $F$ such that $F$ is Frobenius but $F^{\piv}$ is not (\S\ref{subsec:FS-ind-examples}).
\end{remark}

\section{Frobenius-Schur indicators}
\label{sec:FS-ind}

\subsection{Frobenius-Schur indicator}

In a series of papers \cite{MR2313527,MR2381536,MR2366965,MR2725181}, Ng and Schauenburg introduced and studied the $(n, r)$-th Frobenius-Schur (FS) indicator for an object $V$ of a $k$-linear pivotal monoidal category, which we will denote by
\begin{equation*}
  \nu_{n,r}^{\mathrm{NS}}(V)
\end{equation*}
to avoid conflict of notation. In this section, we introduce the $(n,r)$-th FS indicator $\nu_{n,r}(\mathbf{V})$ for a pivotal object in a $k$-linear rigid monoidal category $\mathcal{C}$ and investigate their basic properties.

Let $\mathcal{C}$ be a left rigid monoidal category. Given $X_1, \dotsc, X_n \in \mathcal{C}$, we write
\begin{equation*}
  X_1 \otimes \dotsb \otimes X_n := X_1 \otimes
  (X_2 \otimes (\dotsb \otimes (X_{n - 1} \otimes X_n) \dotsb ))
\end{equation*}
and set $V^{\otimes n} = V_1 \otimes \dotsb \otimes V_n$ with $V_1 = \dotsb = V_n = V$. For an object $\mathbf{V} \in \mathcal{C}^{\piv}$ and a positive integer $n$, we define a bijection
\begin{equation*}
  E_{\mathbf{V}}^{(n)}: \Hom_{\mathcal{C}}(\unitobj, V^{\otimes n})
  \to \Hom_{\mathcal{C}}(\unitobj, V^{\otimes n})
\end{equation*}
as follows: First, we recall that there is a natural isomorphism
\begin{equation*}
  D_{V,W}: \Hom_{\mathcal{C}}(\unitobj, V \otimes W)
  \xrightarrow{\quad \cong \quad} \Hom_{\mathcal{C}}(V^*, W)
  \xrightarrow{\quad \cong \quad} \Hom_{\mathcal{C}}(\unitobj, W \otimes V^{**})
\end{equation*}
for $V, W \in \mathcal{C}$; see \cite[Definition 3.3]{MR2381536}. The map $E_{\mathbf{V}}^{(n)}$ is defined by
\begin{equation*}
  E_{\mathbf{V}}^{(n)}(f) = a^{(n)} \circ (\id_{V^{\otimes (n-1)}} \otimes \phi_V^{-1}) \circ D_{V,V^{\otimes (n-1)}}(f)
\end{equation*}
for $f \in \Hom_{\mathcal{C}}(\unitobj, V^{\otimes n})$, where $a^{(n)}: V^{\otimes (n - 1)} \otimes V \to V^{\otimes n}$ is the isomorphism obtained by using the associativity constraint (which is uniquely determined by Mac Lane's coherence theorem). Strictly speaking, this definition does not make sense in the case where $n = 1$. However, modifying the definition in an obvious way, we obtain
\begin{equation*}
  E_{\mathbf{V}}^{(1)}(f) = \phi_V^{-1} \circ f^{**} \circ \delta_0,
  \quad (f \in \Hom_{\mathcal{C}}(\unitobj, V)),
\end{equation*}
where $\delta_0: \unitobj \to \unitobj^{**}$ is the monoidal structure of $(-)^{**}$.

Now let $\mathcal{C}$ be a $k$-linear monoidal category. Then, since the map $E_{\mathbf{V}}^{(n)}$ is $k$-linear, we can make the following definition:

\begin{definition}
  Let $\mathbf{V} \in \mathcal{C}^{\piv}$ be a pivotal object, and let $n$ and $r$ be integers. If $n > 0$, then we define the {\em $(n, r)$-th FS indicator} of $\mathbf{V}$ by
  \begin{equation*}
    \nu_{n,r}(\mathbf{V}) = \Trace\left( (E_{\mathbf{V}}^{(n)})^r \right).
  \end{equation*}
\end{definition}

If $\mathbf{V} = (V, \phi)$, then $\nu_{n,r}(\mathbf{V})$ is also written as $\nu_{n, r}(V, \phi)$.

It is convenient to extend the $(n, r)$-th FS indicator $\nu_{n, r}$ for all $n, r \in \mathbb{Z}$. Motivated by the definition of the generalized FS indicator \cite{MR2725181}, we define
\begin{equation*}
  E_{\mathbf{V}}^{(0)}: \End_{\mathcal{C}}(\unitobj) \to \End_{\mathcal{C}}(\unitobj),
  \quad f \mapsto \ldim(\mathbf{V}) \circ f
\end{equation*}
for $\mathbf{V} \in \mathcal{C}^{\piv}$ and set $\nu_{0, r}(\mathbf{V}) = \Trace((E_{\mathbf{V}}^{(0)})^r)$ for $r \ge 0$. Now one can uniquely extend $\nu_{n, r}$ for all integers $n$ and $r$ so that
\begin{equation}
  \label{eq:FS-(-n,-r)}
  \nu_{n, r}(\mathbf{V}) = \nu_{-n,-r}(\mathbf{V}^*)
\end{equation}
holds for all $n, r \in \mathbb{Z}$. By definition, we have
\begin{equation}
  \label{eq:FS-scalar-mult-phi}
  \nu_{n, r}(V, c \phi) = c^{-r} \cdot \nu_{n, r}(V, \phi)
\end{equation}
for all $(V, \phi) \in \mathcal{C}^{\piv}$, $n, r \in \mathbb{Z}$ and $c \in k^{\times}$.

If $\mathcal{C}$ is pivotal with pivotal structure $\pivot$, then the map $E_{\mathbf{V}}^{(n)}$ with $\mathbf{V} = (V, \pivot_V)$ is nothing but the map used in \cite{MR2381536} to define $\nu_{n,r}^\mathrm{NS}$. Hence we get
\begin{equation}
  \label{eq:FS-pivot-vs-NS}
  \nu_{n,r}(V, \pivot_V) = \nu_{n,r}^\mathrm{NS}(V)
\end{equation}
for all $V \in \mathcal{C}$ and integers $n$ and $r$. In our setting, an isomorphism $\phi_V: V \to V^{**}$ does not need to be a component of a pivotal structure. This is a small but crucial difference between our theory and \cite{MR2381536}.

Now let $F: \mathcal{C} \to \mathcal{D}$ be a $k$-linear strong monoidal functor between $k$-linear left rigid monoidal categories. For each $n \ge 1$, there is a canonical isomorphism
\begin{equation*}
  F_n(V_1, \dotsc, V_n): F(V_1) \otimes \dotsb \otimes F(V_n) \to F(V_1 \otimes \dotsb \otimes V_n)
  \quad (V_1, \dotsc, V_n \in \mathcal{C})
\end{equation*}
obtained by using the monoidal structure of $F$ in an obvious way. In a similar way as \cite[\S4]{MR2381536}, we can prove the following theorem:

\begin{theorem}
  Let $\mathbf{V} \in \mathcal{C}^{\piv}$. Then the map
  \begin{equation*}
    \widetilde{F}_n: \Hom_{\mathcal{C}}(\unitobj, V^{\otimes n})
    \to \Hom_{\mathcal{C}}(\unitobj, F(V)^{\otimes n}),
    \quad f \mapsto  F_n(V, \cdots, V)^{-1} \circ F(f) \circ F_0^{-1}
  \end{equation*}
  makes the following diagram commute:
  \begin{equation*}
    \begin{CD}
      \Hom_{\mathcal{C}}(\unitobj, V^{\otimes n})
      @>{\widetilde{F}_n}>> \Hom_{\mathcal{C}}(\unitobj, F(V)^{\otimes n}) \\
      @V{E_{\mathbf{V}}^{(n)}}V{}V
      @V{}V{\text{$E_{\mathbf{X}}^{(n)}$ with $\mathbf{X} = F^{\piv}(\mathbf{V})$}}V \\
      \Hom_{\mathcal{C}}(\unitobj, V^{\otimes n})
      @>{\widetilde{F}_n}>> \Hom_{\mathcal{C}}(\unitobj, F(V)^{\otimes n})
    \end{CD}
  \end{equation*}
\end{theorem}

As an immediate consequence, we obtain:

\begin{theorem}
  \label{thm:FS-ind-invariance}
  Let $F: \mathcal{C} \to \mathcal{D}$ be a $k$-linear strong monoidal functor between $k$-linear left rigid monoidal categories. If $F$ is fully faithful, then we have
  \begin{equation*}
    \nu_{n,r}(\mathbf{V}) = \nu_{n,r}(F^{\piv}(\mathbf{V}))
  \end{equation*}
  for all integers $n$ and $r$ and all objects $\mathbf{V} \in \mathcal{C}^{\piv}$.
\end{theorem}

\subsection{Graphical presentation}
\label{subsec:graph-pres}

By Theorem~\ref{thm:FS-ind-invariance} and Mac Lane's coherence theorem, we may assume that $\mathcal{C}$ is strict when we study general properties of the FS indicators. Moreover, by the following lemma, we may also assume that the duality functor is strict:

\begin{lemma}
  \label{lem:strict-rigid}
  For every rigid monoidal category $\mathcal{C}$, there exists a rigid monoidal category $\mathcal{C}^{\str}$ satisfying the following conditions:
  \begin{itemize}
  \item [(S1)] $\mathcal{C}$ is equivalent to $\mathcal{C}^{\str}$ as a monoidal category.
  \item [(S2)] $\mathcal{C}^{\str}$ is a strict monoidal category.
  \item [(S3)] The left duality functor $(-)^*$ for $\mathcal{C}^{\str}$ is a strict monoidal functor.
  \item [(S4)] The right duality functor ${}^*(-)$ for $\mathcal{C}^{\str}$ is the inverse of $(-)^*$.
  \end{itemize}
\end{lemma}
\begin{proof}
  Our proof is based on the argument of Ng and Schauenburg \cite[\S2]{MR2381536}, where a strictification of a pivotal monoidal category is considered. Let $\mathcal{C}$ be a rigid monoidal category. Without loss of generality, we may assume that $\mathcal{C}$ is a strict monoidal category. We define a strict monoidal category $\mathcal{C}^{\str}$ as follows: Objects of $\mathcal{C}^{\str}$ are sequences of length $r \in \mathbb{Z}_{\ge 0}$, like
  \begin{equation}
    \label{eq:C-str-object}
    \mathbb{X} = ((X_1, n_1), \dotsc, (X_r, n_r)),
  \end{equation}
  where $X_i$'s are objects of $\mathcal{C}$ and $n_i$'s are integers. For such an $\mathbb{X}$, we set
  \begin{equation*}
    [\mathbb{X}] = X_1^{*n_1} \otimes \dotsb \otimes X_r^{*n_r},
  \end{equation*}
  where $X^{*0} = X$, $X^{*(n+1)} = (X^{*n})^*$ if $n > 0$, and $X^{*(n-1)} = {}^*(X^{*n})$ if $n < 0$. For the sequence $\mathbb{X} = \emptyset$ of length zero, we understand $[\mathbb{X}] = \unitobj$. For objects $\mathbb{X}, \mathbb{Y} \in \mathcal{C}^{\str}$, the set of morphisms between them is defined by
  \begin{equation*}
    \Hom_{\mathcal{C}^{\str}}(\mathbb{X}, \mathbb{Y}) = \Hom_{\mathcal{C}}([\mathbb{X}], [\mathbb{Y}]),
  \end{equation*}
  and the composition of morphisms in $\mathcal{C}^{\str}$ is defined by the composition in $\mathcal{C}$. It is easy to see that the category $\mathcal{C}^{\str}$ is a strict monoidal category with tensor product given by the concatenation of sequences and unit object $\emptyset$.

  The assignment $\mathbb{X} \mapsto [\mathbb{X}]$ naturally extends to a strong monoidal functor from $\mathcal{C}^{\str}$ to $\mathcal{C}$. This functor is in fact a monoidal equivalence, since it is fully faithful and surjective on objects. Hence $\mathcal{C}^{\str}$ fulfills the conditions (S1) and (S2). We now choose a left and a right dual object for each object of $\mathcal{C}^{\str}$ so that the conditions (S3) and (S4) are satisfied. For $\mathbb{X} \in \mathcal{C}^{\str}$ as in~\eqref{eq:C-str-object}, we set
  \begin{align*}
    \mathbb{X}^* & = ((X_r, n_r+1), \dotsc, (X_1, n_1+1))
  \end{align*}
  and define $\eval_{\mathbb{X}}: \mathbb{X}^* \otimes \mathbb{X} \to \emptyset$ and $\coev_{\mathbb{X}}: \emptyset \to \mathbb{X} \otimes \mathbb{X}^*$ inductively as follows: If $\mathbb{X} = \emptyset$ is the object of length zero, then set $\eval_{\mathbb{X}} = \coev_{\mathbb{X}} = \id_{\unitobj}$. If $\mathbb{X} = ((X, n))$ is an object of $\mathcal{C}^{\str}$ of length one, then set
  \begin{equation*}
    \eval_{\mathbb{X}} = 
    \begin{cases}
      \eval_{X^{*n}} & \text{if $n \ge 0$}, \\
      \eval'_{X^{*(n+1)}} & \text{otherwise},
    \end{cases}
    \text{\quad and \quad}
    \coev_{\mathbb{X}} = 
    \begin{cases}
      \coev_{X^{*n}} & \text{if $n_1 \ge 0$}, \\
      \coev'_{X^{*(n+1)}} & \text{otherwise}.
    \end{cases}
  \end{equation*}
  If $\mathbb{X} = \mathbb{Y} \otimes \mathbb{Z}$ for some $\mathbb{Y}, \mathbb{Z} \in \mathcal{C}^{\str}$ such that the morphisms $\eval_{\mathbb{Y}}$, $\eval_{\mathbb{Z}}$, $\coev_{\mathbb{Y}}$ and $\coev_{\mathbb{Z}}$ have been defined, then set
  \begin{equation}
    \label{eq:C-str-left-dual}
    \eval_{\mathbb{X}} = \eval_{\mathbb{Z}} \circ (\id_{\mathbb{Z}^*} \otimes \eval_{\mathbb{Y}} \otimes \id_{\mathbb{Z}}),
    \quad \coev_{\mathbb{X}} = (\id_{\mathbb{Z}^*} \otimes \coev_{\mathbb{Y}} \otimes \id_{\mathbb{Z}}) \circ \coev_{\mathbb{Z}}.
  \end{equation}
  One can check that the above rule defines $\eval_{\mathbb{X}}$ and $\coev_{\mathbb{X}}$ consistently and the triple $(\mathbb{X}^*, \eval_{\mathbb{X}}, \coev_{\mathbb{X}})$ is a left dual object of $\mathbb{X}$. By \eqref{eq:C-str-left-dual}, the left duality functor for $\mathcal{C}^{\str}$ is strict, {\it i.e.}, the condition (S3) holds.

  Finally, we set ${}^*\mathbb{X} = ((X_r, n_r-1), \dotsc, (X_1, n_1-1))$ for $\mathbb{X} \in \mathcal{C}^{\str}$ as in \eqref{eq:C-str-object}. The condition (S4) is satisfied if we choose the triple $({}^*\mathbb{X}, \eval_{{}^*\mathbb{X}}, \coev_{{}^*\mathbb{X}})$ as a right dual object for $\mathbb{X} \in \mathcal{C}^{\str}$.
\end{proof}

Now let $\mathcal{C}$ be a $k$-linear rigid monoidal category satisfying the conditions (S1)--(S4) of the above lemma. Then we can freely use the graphical technique to express morphisms in $\mathcal{C}$. Our convention is that the source of a morphism is placed on the top of the diagram, and the target is placed on the bottom. For example,
\begin{equation*}
  \includegraphics{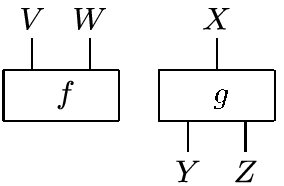}
\end{equation*}
represents the tensor product $f \otimes g$ of $f: V \otimes W \to \unitobj$ and $g: X \to Y \otimes Z$. The evaluation and the coevaluation are expressed by
\begin{equation*}
  \begin{array}{c} \includegraphics{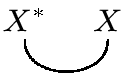} \end{array}
  : X^* \otimes X \to \unitobj
  \quad \text{and} \quad
  \begin{array}{c} \includegraphics{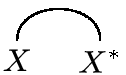} \end{array}
  : \unitobj \to X \otimes X^*,
\end{equation*}
respectively.

\begin{figure}
  \includegraphics{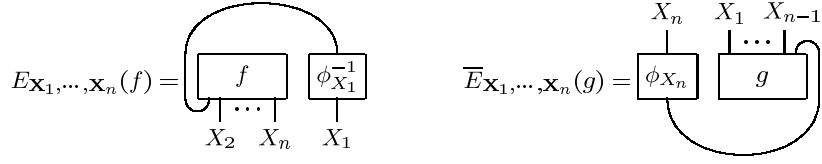}
  \caption{The definition of $E_{\mathbf{X}_1, \cdots, \mathbf{X}_n}$ and $\overline{E}_{\mathbf{X}_1, \cdots, \mathbf{X}_n}$}
  \label{fig:def-E-map}
\end{figure}

For pivotal objects $\mathbf{X}_1, \dotsc, \mathbf{X}_n \in \mathcal{C}^{\piv}$, we define
\begin{gather*}
  E_{\mathbf{X}_1, \dotsc, \mathbf{X}_n}:
  \Hom_{\mathcal{C}}(\unitobj, X_1 \otimes \dotsb \otimes X_n)
  \to \Hom_{\mathcal{C}}(\unitobj, X_2 \otimes \dotsb \otimes X_n \otimes X_1), \\
  \overline{E}_{\mathbf{X}_1, \dotsc, \mathbf{X}_n}:
  \Hom_{\mathcal{C}}(X_1 \otimes \dotsb \otimes X_n, \unitobj)
  \to \Hom_{\mathcal{C}}(X_n \otimes X_1 \dotsb \otimes X_{n - 1}, \unitobj)
\end{gather*}
by Figure~\ref{fig:def-E-map}. By using their graphical expressions, one can check that the diagram
\begin{equation*}
  \begin{CD}
    \Hom_{\mathcal{C}}(X_1 \otimes \dotsb \otimes X_n, \unitobj)
    @>{\overline{E}_{\mathbf{X}_1, \dotsc, \mathbf{X}_n}}>>
    \Hom_{\mathcal{C}}(X_n \otimes X_1 \otimes \dotsb \otimes X_{n - 1}, \unitobj) \\
    @V{(-)^*}V{}V @V{}V{(-)^*}V \\
    \Hom_{\mathcal{C}}(\unitobj, X_n^* \otimes \dotsb \otimes X_1^*)
    @>{E_{\mathbf{X}_n^*, \dotsc, \mathbf{X}_1^*}}>>
    \Hom_{\mathcal{C}}(\unitobj, X_{n-1}^* \otimes \dotsb \otimes X_{n}^* \otimes X_1^*)
  \end{CD}
\end{equation*}
commutes. Note that the map $E_{\mathbf{V}}^{(n)}: \Hom_{\mathcal{C}}(\unitobj, V^{\otimes n}) \to \Hom_{\mathcal{C}}(\unitobj, V^{\otimes n})$ used to define the FS indicator coincides with the map $E_{\mathbf{X}_1, \dotsc, \mathbf{X}_n}$ with $\mathbf{X}_1 = \dotsb = \mathbf{X}_n = \mathbf{V}$. As a consequence, for all $n, r \in \mathbb{Z}$ such that $n > 0$, we have
\begin{equation}
  \label{eq:FS-dual}
  \nu_{n,r}(\mathbf{V}^*)
  = \Trace \Big( (\overline{E}_{\mathbf{V}}^{(n)})^r \Big),
\end{equation}
where $\overline{E}_{\mathbf{V}}^{(n)} = \overline{E}_{\mathbf{X}_1, \dotsc, \mathbf{X}_n}$ with $\mathbf{X}_1 = \dotsb = \mathbf{X}_n = \mathbf{V}$. 

\subsection{Example I. Trace invariants}
\label{sec:example-trace-inv}

Let $\mathcal{C}$ be a $k$-linear left rigid monoidal category. For an integer $n$ and an object $V \in \mathcal{C}$ such that
\begin{equation}
  \label{eq:Hom(V,V**)}
  V \cong V^{**}
  \text{\quad and \quad}
  \End_{\mathcal{C}}(V) \cong k,
\end{equation}
we choose an isomorphism $\phi: V \to V^{**}$ and set
\begin{equation*}
  \mu_n(V) = \nu_{n, 1}(V, \phi) \cdot \nu_{0,-1}(V, \phi) \in k.
\end{equation*}
$\mu_n(V)$ does not depend on the choice of $\phi$ by \eqref{eq:FS-scalar-mult-phi}.

Now we consider the case where $\mathcal{C} = \leftmod{H}$ for some finite-dimensional Hopf algebra $H$ over $k$. For a self-dual absolutely simple left $H$-module $V$, {\em Jedwab's trace invariant} $\mu(V)$  \cite{MR2724230} is defined to be the trace of
\begin{equation*}
  S_V: H / I \to H / I,
  \quad h + I \mapsto S(h) + I
  \quad (h \in H),
\end{equation*}
where $S: H \to H$ is the antipode of $H$ and $I = \{ h \in H \mid h v = 0 \text{ for all } v \in V \}$ is the annihilator of $V$. Now we claim:

\begin{theorem}
  $\mu(V) = \mu_2(V)$.
\end{theorem}
\begin{proof}
  Let $\{ v_i \}_{i = 1}^m$ be a basis of $V$. We denote by $\{ v_i^* \}$ and $\{ v_i^{**} \}$ the basis of $V^*$ and $V^{**}$ dual to $\{ v_i \}$ and $\{ v_i^* \}$, respectively. Fix an isomorphism $p: V \to V^*$ of $H$-modules and let $P$ be the matrix representing $p$ with respect to the basis $\{ v_i \}$ and $\{ v_i^* \}$. Then $\mu(V)$ is equal to the trace of $Q = {}^t (P^{-1}) \cdot P$, where ${}^t(-)$ means the matrix transpose (see \cite{MR2724230}).

  Now we set $\mathbf{V} = (V, \phi)$, where $\phi := (p^{-1})^* \circ p$. Note that the matrix $Q$ represents the map $\phi$ with respect to the basis $\{ v_i \}$ and $\{ v_i^{**} \}$. By Remark~\ref{rem:piv-cov-pdim}, we have
  \begin{equation*}
    \nu_{0, -1}(V, \phi)
    = \eval_{V^*} \circ (\phi \otimes \id) \circ \coev_V
    = \Trace(Q) = \mu(V).
  \end{equation*}
  To compute $\nu_{2,1}(\mathbf{V})$, we consider the following map:
  \begin{equation*}
    T: \Hom_H(V^*, V) \to \Hom_H(V^*, V),
    \quad f \mapsto \phi^{-1} \circ f^*
    \quad (f \in \Hom_H(V^*, V)).
  \end{equation*}
  By the graphical expression of $E_{\mathbf{V}}^{(2)}$, we see that $T$ is conjugate to $E_{\mathbf{V}}^{(2)}$ via the canonical isomorphism $\Hom_H(k, V \otimes V) \cong \Hom_H(V^*, V)$. Since $T(p^{-1}) = p^{-1}$, and since $\Hom_H(V^*, V)$ is one-dimensional, we have
  \begin{equation*}
    \nu_{2,1}(\mathbf{V}) = \Trace(E_{\mathbf{V}}^{(2)}) = \Trace(T) = 1.
  \end{equation*}
  We now conclude $\mu_2(V) = \nu_{2, 1}(V, \phi) \cdot \nu_{0,-1}(V, \phi) = \mu(V)$.
\end{proof}

Let $F: \mathcal{C} \to \mathcal{D}$ be a $k$-linear fully faithful strong monoidal functor between left rigid monoidal categories $\mathcal{C}$ and $\mathcal{D}$. If $V \in \mathcal{C}$ is an object satisfying~\eqref{eq:Hom(V,V**)}, then so does $F(V)$. By the invariance of the FS indicators, we have:

\begin{theorem}
  Let $F: \mathcal{C} \to \mathcal{D}$ be as above. Then
  \begin{equation*}
    \mu_n(F(V)) = \mu_n(V)
  \end{equation*}
  for all $n \in \mathbb{Z}$ and all $V \in \mathcal{C}$ satisfying~\eqref{eq:Hom(V,V**)}.
\end{theorem}

The case where $n = 2$ recovers \cite[Theorem 4.1]{KMN09}.

\subsection{Example II. Indicators of Hopf algebras}
\label{subsec:FS-ind-examples}

Let $H$ be a finite-dimensional Hopf algebra over a field $k$ with comultiplication $\Delta$ and antipode $S$. Recall that the {\em convolution product} $\star$ in $\End_k(H)$ is defined by
\begin{equation*}
  (f \star g)(h) = f(h_{(1)}) g(h_{(2)})
  \quad (f, g \in \End_k(H), h \in H),
\end{equation*}
where $h_{(1)} \otimes h_{(2)} = \Delta(h)$ is the comultiplication of $h \in H$ in the Sweedler notation. Let $n$ be a positive integer. The {\em $n$-th indicator} of $H$, introduced by Kashina, Montgomery and Ng in \cite{KMN09}, is defined by
\begin{equation*}
  \nu_n^\mathrm{KMN}(H) := \Trace(S \circ \id_H^{\star (n - 1)}),
\end{equation*}
where $(-)^{\star m}$ means the $m$-th power with respect to $\star$. Note that the identity map is invertible with respect to $\star$ (with inverse $S$). Hence the $n$-th indicator can be defined for all integers $n$ by the same formula as above. For example,
\begin{equation*}
  \nu_0^{\rm KMN}(H) = \Trace(S^2).
\end{equation*}

The aim of this subsection is to explain that the $n$-th indicator can be expressed as the $(n, 1)$-th FS indicator of a certain pivotal object. We consider the map
\begin{equation*}
  \phi: H \to H^{**},
  \quad \langle \phi(h), p \rangle = \langle p, S^2(h) \rangle
  \quad (h \in H, p \in H^*).
\end{equation*}
Then $\mathbf{R}_H := (H, \phi)$ is an object of $(\leftmod{H})^{\piv}$ (if we regard $H$ as a left $H$-module by the multiplication from the left). Our claim is:

\begin{theorem}
  \label{thm:FS-ind-and-KMN-ind}
  $\nu_n^{\rm KMN}(H) = \nu_{n, 1}(\mathbf{R}_H^*)$ for all integers $n$.
\end{theorem}
\begin{proof}
  Write $\mathbf{R} = \mathbf{R}_H$. For $n = 0$, the claim is proved by Remark~\ref{rem:piv-cov-pdim}, as follows:
  \begin{equation*}
    \nu_{0,1}(\mathbf{R}^*) = \underline{\dim}^{(\ell)}(\mathbf{R}^*)
    = \underline{\dim}^{(r)}(\mathbf{R})
    = \Trace(S^2) = \nu_0^{\rm KMN}(H).
  \end{equation*}
  For $n > 0$, we define $\mathcal{E}_H^{(n)}: \Hom_H(H^{\otimes n}, k) \to \Hom_H(H^{\otimes n}, k)$ by
  \begin{equation*}
    \mathcal{E}_H^{(n)}(f)(h_1 \otimes \dotsb \otimes h_n) = f(h_2 \otimes \dotsb \otimes h_n \otimes S^2(h_1))
  \end{equation*}
  for $f \in \Hom_H(H^{\otimes n}, k)$ and $h_1, \dotsc, h_n \in H$. We have showed that $\nu_n^{\rm KMN}(H)$ is equal to the trace of this map in \cite[Lemma 3.2]{2011arXiv1106.2936S}. On the other hand, this map coincides with the map $\overline{E}_{\mathbf{R}}^{(n)}$ introduced in \S\ref{subsec:graph-pres}. Hence, by \eqref{eq:FS-dual}, we conclude
  \begin{equation*}
    \nu_{n, 1}(\mathbf{R}^*) = \Trace(\overline{E}_{\mathbf{R}}^{(n)}) = \Trace(\mathcal{E}_H^{(n)}) = \nu_n^{\rm KMN}(H).
  \end{equation*}
  Finally, we suppose that $n < 0$. Set $c = \nu_{-1}^{\rm KMN}(H)$, and let $\Lambda \in H$ be a non-zero left integral. We recall from \cite[Proposition 3.13]{2011arXiv1106.2936S} that
  \begin{equation}
    \label{eq:KMN-ind-(-1)}
    S^2(\Lambda) = c \Lambda.
  \end{equation}
  Note that $-n > 0$. By \cite[Proposition 3.16]{2011arXiv1106.2936S} and the above argument, we have
  \begin{equation}
    \label{thm:FS-ind-and-KMN-ind-proof-1}
    \nu_{n}^{\rm KMN}(H)
    = c \cdot \Trace \Big( (\mathcal{E}_H^{(-n)})^{-1} \Big)
    = c \cdot \Trace \Big( (\overline{E}_{\mathbf{R}}^{(-n)})^{-1} \Big)
    = c \cdot \nu_{-n, -1}(\mathbf{R}^*).
  \end{equation}
  By~\eqref{eq:KMN-ind-(-1)}, $c$ is invertible. Hence, if we could prove
  \begin{equation}
    \label{eq:regular-obj-dual}
    \mathbf{R}^* \cong (H, c^{-1} \phi),
  \end{equation}
  then, by~\eqref{eq:FS-(-n,-r)}, \eqref{eq:FS-scalar-mult-phi} and~\eqref{thm:FS-ind-and-KMN-ind-proof-1}, we would have
  \begin{equation*}
    \nu_{n, 1}(\mathbf{R}^*)
    = \nu_{n, 1}(H, c^{-1} \phi)
    = c \cdot \nu_{n, 1}(\mathbf{R})
    = c \cdot \nu_{-n,-1}(\mathbf{R}^*) = \nu_{n}^{\rm KMN}(H).
  \end{equation*}
  Let us prove~\eqref{eq:regular-obj-dual} to complete the proof. Following Radford \cite{MR1265853}, the map
  \begin{equation*}
    \psi: H^* \to H, \quad \psi(p) = \langle p, \Lambda_{(1)} \rangle \Lambda_{(2)} \quad(p \in H^*)
  \end{equation*}
  is an isomorphism of left $H$-modules. We will show that $\psi$ is in fact an isomorphism from $\mathbf{R}^*$ to $(H, c^{-1}\phi)$ in $(\leftmod{H})^{\piv}$, that is,
  \begin{equation}
    \label{thm:FS-ind-and-KMN-ind-proof-2}
    \psi^{**} (\phi^{-1})^* = c^{-1} \phi \psi.
  \end{equation}
  For $p, q \in H^*$, we have
  \begin{equation*}
    \langle \psi^*(q), p \rangle
    = \langle q, \psi(p) \rangle
    = \langle p, \Lambda_{(1)} \rangle \langle q, \Lambda_{(2)} \rangle
    = \langle \phi(S^{-2}(\Lambda_{(1)})), p \rangle \langle q, \Lambda_{(2)} \rangle.
  \end{equation*}
  This implies $\phi^{-1} \psi^*(q) = S^{-2}(\Lambda_{(1)}) \langle q, \Lambda_{(2)} \rangle$. By~\eqref{eq:KMN-ind-(-1)}, we also have
  \begin{equation*}
    S^{-2}(\Lambda_{(1)}) \otimes \Lambda_{(2)}
    = (\id_H \otimes S^2) \Delta(S^{-2}(\Lambda))
    = c^{-1} \Lambda_{(1)} \otimes S^2(\Lambda_{(2)}).
  \end{equation*}
  By the above results, \eqref{thm:FS-ind-and-KMN-ind-proof-2} is proved as follows:
  \begin{align*}
    \langle \psi^{**} (\phi^{-1})^*(p), q \rangle
    & = \langle p, \phi^{-1} \psi^*(q) \rangle \\
    & = \langle p, S^{-2}(\Lambda_{(1)}) \rangle \langle q, \Lambda_{(2)}) \rangle \\
    & = c^{-1} \langle p, \Lambda_{(1)} \rangle \langle q, S^2(\Lambda_{(2)}) \rangle \\
    & = c^{-1} \langle q, S^2 \psi(p) \rangle \\
    & = c^{-1} \langle \phi \psi(p), q \rangle. \qedhere
  \end{align*}
\end{proof}

\begin{remark}
  The following assertions are equivalent:
  \begin{enumerate}
  \item $c = \nu_{-1}^{\rm KMN}(H)$ is equal to $1$.
  \item $\mathbf{R}_H \cong \mathbf{R}_H^*$ in $(\leftmod{H})^{\piv}$.
  \end{enumerate}
  Indeed, by \eqref{eq:regular-obj-dual}, it is obvious that (1) implies (2). To show that the converse holds, we note that every $H$-linear map $H \to H$ is of the form $t_a(h) = h a$ for some $a \in H$. For $a \in H$, one can check that:
  \begin{equation*}
    c^{-1} \phi \circ t_a = (t_a)^{**} \circ \phi
    \quad \iff
    \quad S^2(h a) = c S^2(h) a
    \quad \text{for all $h \in H$}.
  \end{equation*}
  Thus if $\mathbf{R}_H$ and $\mathbf{R}_H^* \cong (H, c^{-1} \phi)$ are isomorphic in $(\leftmod{H})^{\piv}$, then there exists an invertible element $a \in H$ satisfying the above equivalent conditions. Applying the counit and substituting $h = 1$, we conclude that such an element $a \in H^{\times}$ does not exist unless $c = 1$. Therefore (2) implies (1).
\end{remark}

\begin{remark}
  Let $\mathcal{V}$ denote the category of finite-dimensional vector spaces over $k$, and let $U: \leftmod{H} \to \mathcal{V}$ be the forgetful functor. Then $I_{\ell} = H \otimes_k (-)$ is a left adjoint functor of $U$ with unit and counit given by
  \begin{equation*}
    \renewcommand{\arraystretch}{1.25}
    \begin{array}{ccc}
      \eta_V: V \to U I_{\ell}(V), & v \mapsto 1 \otimes v & (v \in V \in \mathrm{vect}_k) \\
      \varepsilon_V: I_{\ell} U(M) \to M, & h \otimes m \mapsto h m & (h \in H, m \in M \in \leftmod{H}).
    \end{array}
  \end{equation*}
  The left adjoint functor of $U^{\piv}$ constructed in Theorem~\ref{thm:piv-cov-induced-adj} is given by
  \begin{equation*}
    I_{\ell}^{\piv}((V, f)) = (H \otimes_k V, \phi \otimes f)
  \end{equation*}
  for $(V, f) \in \mathcal{V}^{\piv}$. Hence, in particular, $\mathbf{R}_H \cong I_{\ell}^{\piv}(\mathbf{1})$, where $\mathbf{1}$ is the unit object of $\mathcal{V}^{\piv}$. Equation~\eqref{eq:intro-KMN} in Introduction is obtained by Theorem~\ref{thm:FS-ind-and-KMN-ind} and \eqref{eq:piv-cov-induced-adj-dual}. Note that $U$ is Frobenius, since a finite-dimensional Hopf algebra is a Frobenius algebra. On the other hand, as we have seen, $\mathbf{R}_H$ is not self-dual in general. Thus the functor $U^{\piv}$ is not Frobenius in general ({\it cf}. Remark~\ref{rem:F-piv-Frobenius}).
\end{remark}

Let $H$ and $K$ be two finite-dimensional Hopf algebras over $k$. Suppose that there exists an equivalence $F: \leftmod{H} \to \leftmod{K}$ of $k$-linear monoidal categories. We are interested in the following question: 
\begin{equation}
  \label{eq:KMN-question}
  \text{Is $F^{\piv}(\mathbf{R}_H)$ isomorphic to $\mathbf{R}_K$?}
\end{equation}
A positive answer to this question, together with Theorem~\ref{thm:FS-ind-invariance}, would give a conceptual explanation why the $n$-th indicator is a gauge invariant \cite{KMN09}. We conclude this section by noting the following result:

\begin{proposition}
  The answer to \eqref{eq:KMN-question} is yes, provided that $H$ and $K$ are semisimple and $k$ is of characteristic zero.
\end{proposition}
\begin{proof}
  For an object $V$ of a pivotal monoidal category, set $\mathbf{P}_V = (V, \mathsf{p}_V)$, where $\mathsf{p}$ is the pivotal structure.
  Let $F$ be as above. Since $F(H) \cong K$ as $K$-modules, and since $F$ preserves the canonical pivotal structure \cite[Corollary 6.2]{MR2381536}, we have
  \begin{equation*}
    F^{\piv}(\mathbf{R}_H) = F^{\piv}(\mathbf{P}_H) \cong \mathbf{P}_{F(H)} \cong \mathbf{P}_K = \mathbf{R}_K. \qedhere
  \end{equation*}
\end{proof}

\section{Indicators of the adjoint object}
\label{sec:ind-adj-obj}

\subsection{The adjoint object}

Given a (strict) monoidal category $\mathcal{C}$, we denote by $\mathcal{Z}(\mathcal{C})$ the center of $\mathcal{C}$. Thus, an object of $\mathcal{Z}(\mathcal{C})$ is a pair $(M, \sigma_M)$ consisting of an object $M \in \mathcal{C}$ and a {\em half-braiding} for $M$, {\it i.e.}, a natural isomorphism
\begin{equation*}
  \sigma_M: M \otimes (-) \to (-) \otimes M
\end{equation*}
satisfying $\sigma_M(V \otimes W) = (\id_V \otimes \sigma_M(W)) \circ (\sigma_M(V) \otimes \id_W)$ for all $V, W \in \mathcal{C}$, and a morphism $(M, \sigma_M) \to (N, \sigma_N)$ in $\mathcal{Z}(\mathcal{C})$ is a morphism $M \to N$ in $\mathcal{C}$ compatible with the half-braiding. We note that if $\mathcal{C}$ is rigid, then so is $\mathcal{Z}(\mathcal{C})$. More precisely, one can choose a left dual object for each $(M, \sigma_M) \in \mathcal{Z}(\mathcal{C})$ so that the duality transformation of the forgetful functor
\begin{equation}
  \label{eq:center-forg}
  U: \mathcal{Z}(\mathcal{C}) \to \mathcal{C}, \quad (M, \sigma_M) \mapsto M
\end{equation}
is the identity.

Let $\mathcal{C}$ be a finite tensor category over an algebraically closed field $k$. Then the functor \eqref{eq:center-forg} has a right adjoint functor \cite{MR2119143}. Hence, by Theorem~\ref{thm:piv-cov-induced-adj}, $U^{\piv}$ has a right adjoint functor, say $I_r^{\piv}$. We now set
\begin{equation*}
  \mathbf{A}_{\mathcal{C}} = U^{\piv} I^{\piv}_r(\unitobj) \in \mathcal{C}^{\piv}
\end{equation*}
and write it as $\mathbf{A}_{\mathcal{C}} = (A_{\mathcal{C}}, \phi_{\mathcal{C}})$ or $\mathbf{A} = (A, \phi)$ if $\mathcal{C}$ is clear. The isomorphism class of $\mathbf{A}_{\mathcal{C}}$ and that of $A_{\mathcal{C}}$ do not depend on the choice of $I_r$. We call $\mathbf{A}_{\mathcal{C}}$ (or its underlying object $A_{\mathcal{C}}$) the {\em adjoint object} of $\mathcal{C}$, since, as we will see in \S\ref{subsec:adj-obj-Hopf-case}, $A_{\mathcal{C}}$ is the adjoint representation of $H$ if $\mathcal{C} = \leftmod{H}$ for some finite-dimensional Hopf algebra $H$.

The aim of this section is to investigate the properties of the FS indicators of the adjoint object. We first remark the following relation between $\nu_{n,r}(\mathbf{A}_{\mathcal{C}})$ and $\nu_{n,r}^\mathrm{NS}(A_{\mathcal{C}})$, the $(n, r)$-th FS indicator in the sense of Ng and Schauenburg \cite{MR2381536}.

\begin{theorem}
  \label{thm:adj-obj-ind-NS}
  Suppose that $\mathcal{C}$ has a pivotal structure. Then we have
  \begin{equation*}
    \nu_{n,r}(\mathbf{A}_{\mathcal{C}}) = \nu_{n,r}^\mathrm{NS}(A_{\mathcal{C}})
  \end{equation*}
  for all integers $n$ and $r$. Thus, $\nu_{n,r}^\mathrm{NS}(A_{\mathcal{C}})$ does not depend on the choice of the pivotal structures of $\mathcal{C}$.
\end{theorem}
\begin{proof}
  Let $\pivot$ denote the pivotal structure of $\mathcal{C}$. Then $\mathcal{Z}(\mathcal{C})$ has a unique pivotal structure, which we denote by the same symbol $\pivot$, such that the forgetful functor $U$ preserves the pivotal structure. By Lemma~\ref{lem:piv-preserve-func-adj}, we have
  \begin{equation*}
    \phi_{\mathcal{C}} = (\xi_{\unitobj}^{(r)})^{-1} \circ I_{r}(\delta_0)
    = (\xi_{\unitobj}^{(r)})^{-1} \circ I_{r}(\pivot_{\unitobj}) = \pivot_{I_{r}(\unitobj)},
  \end{equation*}
  where $\xi^{(r)}$ is the natural isomorphism given in Lemma~\ref{lem:monoidal-adj-bidual} and $\delta_0: \unitobj \to \unitobj^{**}$ is the monoidal structure of $(-)^{**}$. Hence the result follows from~\eqref{eq:FS-pivot-vs-NS}.
\end{proof}

The indicators of the adjoint object have the following gauge invariance:

\begin{theorem}
  \label{thm:adj-obj-invariance}
  Let $\mathcal{C}$ and $\mathcal{D}$ be finite tensor categories over $k$. If there exists an equivalence $F: \mathcal{C} \to \mathcal{D}$ of $k$-linear monoidal categories, then we have
  \begin{equation*}
    \nu_{n,r}(\mathbf{A}_{\mathcal{C}}) = \nu_{n,r}(\mathbf{A}_{\mathcal{D}})
  \end{equation*}
  for all integers $n$ and $r$.
\end{theorem}
\begin{proof}
  Let $U_{\mathcal{C}}: \mathcal{Z}(\mathcal{C}) \to \mathcal{C}$ and $U_{\mathcal{D}}: \mathcal{Z}(\mathcal{D}) \to \mathcal{D}$ denote the forgetful functors. Then there exists an equivalence $\widetilde{F}: \mathcal{Z}(\mathcal{C}) \to \mathcal{Z}(\mathcal{D})$ of braided monoidal categories such that $U_{\mathcal{D}} \circ \widetilde{F} = F \circ U_{\mathcal{C}}$. Applying the 2-functor $(-)^{\piv}$, we have
  \begin{equation}
    \label{eq:adj-obj-inv-1}
    U_{\mathcal{D}}^{\piv} \circ \widetilde{F}^{\piv} = F^{\piv} \circ U_{\mathcal{C}}^{\piv}.
  \end{equation}
  Fix a right adjoint $I^{\piv}$ of $U_{\mathcal{C}}^{\piv}$ and a quasi-inverse $\overline{F}$ of $F$. By \eqref{eq:adj-obj-inv-1}, $\widetilde{F}^{\piv} \circ I^{\piv} \circ \overline{F}{}^{\piv}$ is right adjoint to $U_{\mathcal{D}}^{\piv}$. Hence we have isomorphisms
  \begin{equation*}
    \mathbf{A}_{\mathcal{D}}
    \cong U_{\mathcal{D}}^{\piv} \widetilde{F}^{\piv} I^{\piv} \overline{F}{}^{\piv} (\mathbf{1})
    \cong F^{\piv} U_{\mathcal{C}}^{\piv} I^{\piv} (\mathbf{1})
    \cong F^{\piv}(\mathbf{A}_{\mathcal{C}})
  \end{equation*}
  in $\mathcal{D}^{\piv}$. Now the result follows from Theorem~\ref{thm:FS-ind-invariance}.
\end{proof}

\subsection{Hopf monadic description of the center}

For further investigation on the adjoint object, we realize $\mathcal{Z}(\mathcal{C})$ as the category of modules over a certain Hopf monad $Z$ on $\mathcal{C}$. We first recall the notion of {\em coends} to define $Z$. Let, in general, $\mathcal{P}$ and $\mathcal{Q}$ be categories, and let $B: \mathcal{P}^{\op} \times \mathcal{P} \to \mathcal{Q}$ be a functor. A {\em dinatural transformation} $i: B \to C$, where $C \in \mathcal{Q}$, is a family
\begin{equation*}
  i = \{ i(X): B(X, X) \to C \}_{X \in \mathcal{P}}
\end{equation*}
of morphisms in $\mathcal{Q}$ such that $i(X) \circ B(f, X) = i(Y) \circ B(Y, f)$ for all morphisms $f: X \to Y$ in $\mathcal{C}$. A {\em coend} of $B$ is a pair $(C, i)$ of an object $C \in \mathcal{Q}$ and a dinatural transformation $i: B \to C$ satisfying the following universal property: If $i': B \to C'$ is another dinatural transformation, then there uniquely exists a morphism $f: C \to C'$ such that $f \circ i(X) = i'(X)$ for all $X \in \mathcal{P}$. If this is the case, we write
\begin{equation*}
  C = \int^{X \in \mathcal{C}} B(X, X)
\end{equation*}
and call $i$ the {\em universal dinatural transformation}.

Now let $\mathcal{C}$ be a finite tensor category over $k$. For simplicity, we assume that $\mathcal{C}$ fulfills the conditions (S1)--(S4) of Lemma \ref{lem:strict-rigid}. By \cite[Corollary 5.1.8]{MR1862634}, the coend
\begin{equation*}
  Z(V) = \int^{X \in \mathcal{C}} X^* \otimes V \otimes X
\end{equation*}
exists for all $V \in \mathcal{C}$. By the parameter theorem for coends \cite[IX.7]{MR1712872}, we can extend $V \mapsto Z(V)$ as an endofunctor $Z$ on $\mathcal{C}$.

Day and Street \cite{MR2342829} showed that the functor $Z$ has a structure of a monad and the category $\mathcal{Z}(\mathcal{C})$ is isomorphic to $\leftmod{Z}$. Since $\mathcal{Z}(\mathcal{C})$ is a braided rigid monoidal category in such a way that the forgetful functor to $\mathcal{C}$ is strict monoidal, $Z$ has a structure of a quasitriangular Hopf monad. Following Brugui\`res and Virelizier \cite{MR2869176}, the Hopf monad structure of $Z$ is given as follows: First, let
\begin{equation*}
  i_V(X): X^* \otimes V \otimes X \to Z(V)
  \quad (V, X \in \mathcal{C})
\end{equation*}
denote the universal dinatural transformation for the coend $Z(V)$, which will be depicted as in Figure~\ref{fig:universal-dinat-trans-1}. For $V \in \mathcal{C}$ and $X_1, \dotsc, X_n \in \mathcal{C}$, define
\begin{equation}
  \label{eq:universal-dinat-trans-n}
  i_V^{(n)}(X_1, \dotsc, X_n): X_n^* \otimes \dotsb \otimes X_1^* \otimes V \otimes X_1 \otimes \dotsb \otimes X_n \to Z^n(V)
\end{equation}
by Figure~\ref{fig:universal-dinat-trans-n}. By repeated use of Fubini's theorem for coends \cite[IX.8]{MR1712872}, we see that $Z^n(V)$ has a structure of the coend
\begin{equation*}
  \int^{X_1, \dotsc, X_n \in \mathcal{C}}
  X_n^* \otimes \dotsb \otimes X_1^* \otimes V \otimes X_1 \otimes \dotsb \otimes X_n
\end{equation*}
with universal dinatural transformation $i_V^{(n)}$. The Hopf monad structure on the functor $Z$ is now defined by Figure~\ref{fig:Hopf-monad-Z}. We omit the description of the quasitriangular structure of $Z$ since we will not use it; see \cite{MR2869176} for details, where, more generally, the quantum double of a Hopf monad is considered.

\begin{figure}
  \includegraphics{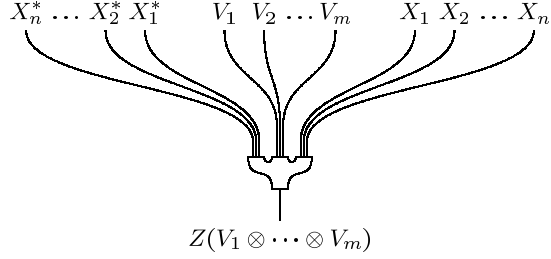}
  \caption{The morphism $i_{V_1 \otimes \dotsb \otimes V_m}(X_1 \otimes \dotsb \otimes X_n)$}
  \label{fig:universal-dinat-trans-1}
\end{figure}

\begin{figure}
  \includegraphics{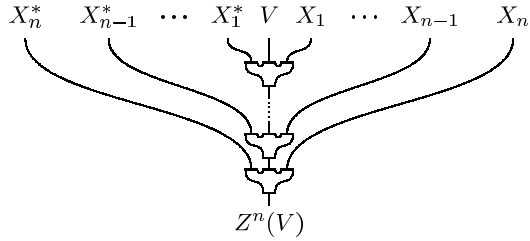}
  \caption{The morphism $i_V^{(n)}(X_1, \dotsc, X_n)$}
  \label{fig:universal-dinat-trans-n}
\end{figure}

\begin{figure}
  \includegraphics{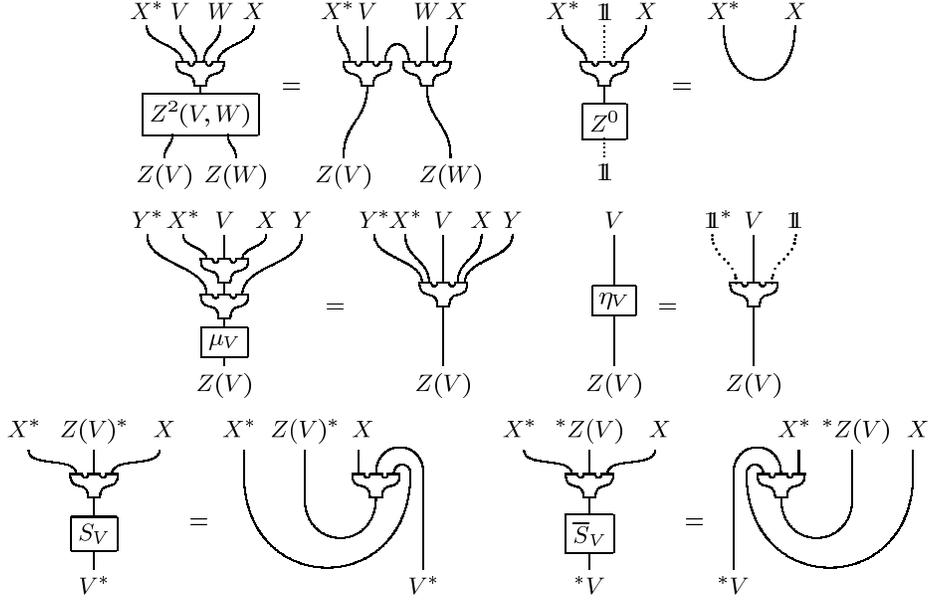}
  \caption{Hopf monad structure on $Z$}
  \label{fig:Hopf-monad-Z}
\end{figure}

Following \cite{MR2342829} and \cite{MR2869176}, an isomorphism $\leftmod{Z} \cong \mathcal{Z}(\mathcal{C})$ is given as follows: First, given an object $M \in \mathcal{C}$, we define a natural transformation $\partial_M$ by
\begin{equation*}
  \begin{CD}
    \partial_M(X):
    M \otimes X @>{\coev \otimes \id}>> X \otimes X^* \otimes M \otimes X
    @>{\id \otimes i_M(X)}>> X \otimes Z(M)
  \end{CD}
\end{equation*}
for $X \in \mathcal{C}$. Then the map
\begin{equation*}
  \Sigma: \Hom_{\mathcal{C}}(Z(M), M) \to \HOM(M \otimes (-), (-) \otimes M),
  \quad a \mapsto (\id_{(-)} \otimes a) \circ \partial_M(-)
\end{equation*}
is an isomorphism of vector spaces. A morphism $a: Z(M) \to M$ makes $M$ into a $Z$-module if and only if $\Sigma(a)$ is a half-braiding for $M$. Hence we get a bijection
\begin{equation}
  \label{eq:Hopf-monad-Z-cat-iso}
  \leftmod{Z} \to \mathcal{Z}(\mathcal{C}), \quad (M, a) \mapsto (M, \Sigma(a))
\end{equation}
on objects. It turns out that this bijection is in fact an isomorphism of monoidal categories (and, moreover, an isomorphism of braided monoidal categories if we endow $Z$ with an appropriate R-matrix).

Let $T: \leftmod{Z} \to \mathcal{C}$ denote the forgetful functor. To identify $\leftmod{Z}$ with $\mathcal{Z}(\mathcal{C})$ via \eqref{eq:Hopf-monad-Z-cat-iso}, we shall check that the following two assertions hold: Under the identification, (i) a right adjoint functor of $U$ corresponds to that of $T$, and (ii) the left dual object in $\mathcal{Z}(\mathcal{C})$ corresponds to the left dual object in $\leftmod{Z}$. The first claim follows since the diagram
\begin{equation}
  \label{eq:Hopf-monad-Z-forg-func}
  \xymatrix{
    \leftmod{Z} \ar[rr]^{\cong}_{\eqref{eq:Hopf-monad-Z-cat-iso}} \ar[rd]_{T}
    & & \mathcal{Z}(\mathcal{C}) \ar[ld]^{U} \\
    & \mathcal{C}
  }
\end{equation}
commutes. To verify (ii), first check that the duality transformations of $U$ and $T$ are identities. Then, by Lemma \ref{lem:dual-trans-composition} and the commutativity of \eqref{eq:Hopf-monad-Z-forg-func}, the duality transformation for \eqref{eq:Hopf-monad-Z-cat-iso} is shown to be the identity.

We now identify $\leftmod{Z}$ with $\mathcal{Z}(\mathcal{C})$ via \eqref{eq:Hopf-monad-Z-cat-iso}. In what follows, a $Z$-module will be denoted by the same symbol as its underlying object. Given $M \in \leftmod{Z}$, we denote by $a_M$ and $\sigma_M$ the action of $Z$ on $M$ and the corresponding half-braiding for $M$, respectively. By definition, we have
\begin{equation}
  \label{eq:Hopf-monad-Z-half-braiding}
  \begin{array}{c} \includegraphics{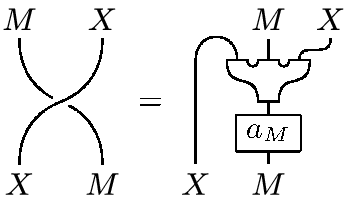} \end{array}
\end{equation}
for all $X \in \mathcal{C}$, where the half-braiding $\sigma_M(X): M \otimes X \to X \otimes M$ is expressed by a crossing in the right-hand side.

\subsection{Hopf monadic description of the adjoint object}

Now we describe the adjoint object $\mathbf{A}_{\mathcal{C}}$ by using the Hopf monad $Z$. From the viewpoint of the Hopf monadic description of the center, a left adjoint functor of the forgetful functor $U: \mathcal{Z}(\mathcal{C}) \to \mathcal{C}$ is easier to deal with than a right one. Indeed, a left adjoint of $U$ is just a free $Z$-module functor
\begin{equation*}
  I_{\ell}: \mathcal{C} \to \leftmod{Z},
  \quad V \mapsto (Z(V), \mu_V)
  \quad (V \in \mathcal{C}).
\end{equation*}
By \eqref{eq:piv-cov-induced-adj-dual}, the adjoint object is given by $\mathbf{A}_{\mathcal{C}} = U^{\piv}(I_{\ell}^{\piv}(\mathbf{1})^*)$, where $I_{\ell}^{\piv}$ is the left adjoint of $U^{\piv}$ constructed in Theorem~\ref{thm:piv-cov-induced-adj}. Thus we first give a description of
\begin{equation}
  \label{eq:dual-adj-obj}
  \mathbf{F}_{\mathcal{C}} := U^{\piv}(I_{\ell}^{\piv}(\mathbf{1})),
\end{equation}
which is the object such that $\mathbf{F}_{\mathcal{C}}^* = \mathbf{A}_{\mathcal{C}}$.

\begin{lemma}
  \label{lem:adj-obj-dual}
  Write $\mathbf{F}_{\mathcal{C}}$ as $\mathbf{F}_{\mathcal{C}} = (F, \psi)$. Then $F = Z(\unitobj)$ and $\psi$ is the isomorphism determined by the following property:
  \begin{equation}
    \label{eq:adj-obj-dual-1}
    \psi \circ i_{\unitobj}(X) = i_{\unitobj}({}^{**}X)^{**}
    \quad (X \in \mathcal{C}).
  \end{equation}
\end{lemma}
\begin{proof}
  $F = Z(\unitobj)$ is obvious from~\eqref{eq:dual-adj-obj}. To prove~\eqref{eq:adj-obj-dual-1}, we compute the isomorphism $\xi^{(\ell)}_V$ given in Lemma \ref{lem:monoidal-adj-bidual}. Since the duality transformation of $U$ is the identity, it is given by the composition
  \begin{equation*}
    \begin{CD}
      \xi^{(\ell)}_V: Z(V^{**}) @>{Z(\eta_V^{**})}>> Z(Z(V)^{**}) @>{a_{Z(V)^{**}}}>> Z(V)^{**},
    \end{CD}
  \end{equation*}
  where $a_{Z(V)^{**}}$ is the action of $Z$ on $(Z(V), \mu_V)^{**}$. Let, in general, $M$ be a $Z$-module with action $a_M$. By \eqref{eq:Hopf-monad-dual-module} and the definition of the left antipode of $Z$, we have
  \begin{align*}
    a_{M^*} \circ i_{M^*}(X)
    & = S_M \circ Z(a_M^*) \circ i_{M^*}(X) \\
    & = S_M \circ i_{M^*}(X) \circ (\id_{X^*} \otimes a_M^* \otimes \id_X)
    = \!\!\! \begin{array}{c} \includegraphics{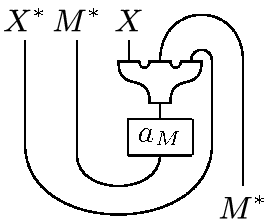} \end{array}
  \end{align*}
  for all $X \in \mathcal{C}$. Replacing $M$ with $M^{*}$, we obtain the following description of the action of $Z$ on $M^{**}$:
  \begin{equation*}
    a_{M^{**}} \circ i_{M^{**}}(X) = (a_M \circ i_M({}^{**}X))^{**}
    \quad (X \in \mathcal{C}).
  \end{equation*}
  Applying this formula to $M = (Z(V), \mu_V)$, we compute:
  \begin{align*}
    \xi^{(\ell)}_V \circ i_{V^{**}}(X)
    & = a_{Z(V)^{**}} \circ Z(\eta_V^{**}) \circ i_{V^{**}}(X) \\
    & = a_{Z(V)^{**}} \circ i_{Z(V)^{**}}(X) \circ (\id_{X^*} \otimes \eta_V^{**} \otimes \id_X) \\
    & = (\mu_V \circ i_{Z(V)}({}^{**}X))^{**} \circ (\id_{{}^* \! X} \otimes \eta_V \otimes \id_{{}^{**} \! X})^{**} \\
    & = (\mu_V \circ i_{Z(V)}({}^{**}X) \circ (\id_{{}^* \! X} \otimes \eta_V \otimes \id_{{}^{**} \! X}))^{**} \\
    & = (\mu_V \circ Z(\eta_V) \circ i_{V}({}^{**}X))^{**} = i_{V}({}^{**}X)^{**}
  \end{align*}
  for all $V, X \in \mathcal{C}$. Since $\psi = \xi^{(\ell)}_{\unitobj}$, we obtain the result.
\end{proof}

\subsection{Relation to the Drinfeld isomorphisms}

For an object $M \in \mathcal{Z}(\mathcal{C})$, the {\em Drinfeld isomorphism} is defined by
\begin{equation*}
  \begin{CD}
    \mathbf{u}_M: M
    @>{\id \otimes \coev}>> M \otimes M^* \otimes M^{**}
    @>{\sigma_M \otimes \id}>> M^* \otimes M \otimes M^{**}
    @>{\eval \otimes \id}>> M^{**},
  \end{CD}
\end{equation*}
where $\sigma_M: M \otimes (-) \to (-) \otimes M$ is the half-braiding of $M$. Here we show the following relation between the adjoint object and the Drinfeld isomorphism:

\begin{proposition}
  \label{prop:adj-obj-Z}
  $\mathbf{A}_{\mathcal{C}} = (A, \mathbf{u}_A)$, where $A = Z(\unitobj)^*$.
\end{proposition}
\begin{proof}
  Let $\mathbf{F}_{\mathcal{C}} = (F, \psi)$ be as in Lemma~\ref{lem:adj-obj-dual}. Then the claim is equivalent to
  \begin{equation*}
    \mathbf{F}_{\mathcal{C}} = (F, \mathbf{v}_F),
  \end{equation*}
  where $\mathbf{v}_F = {}^*(\mathbf{u}_{F^*}^{-1})$. To show this, we first note that $\mathbf{v}_F^{-1}$ is given by
  \begin{equation*}
    \mathbf{v}_F^{-1} = {}^*(\mathbf{u}_{F^*})
    = \begin{array}{c} \includegraphics{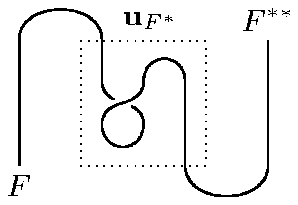} \end{array}
    = \begin{array}{c} \includegraphics{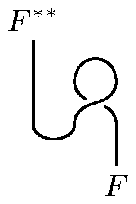} \end{array}
    \mathop{=}^{\eqref{eq:Hopf-monad-Z-half-braiding}}
    \begin{array}{c} \includegraphics{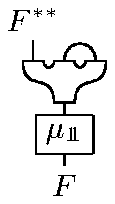} \end{array}.
  \end{equation*}
  Let $X \in \mathcal{C}$ be an object, and write $i = i_{\unitobj}({}^{**}X)$. Then, by Lemma~\ref{lem:adj-obj-dual},
  \begin{align*}
    & \mathbf{v}_F^{-1} \circ \psi \circ i_{\unitobj}(X) \\[1.5ex]
    & = \!\!\! \begin{array}{c} \includegraphics{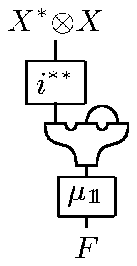} \end{array}
    = \!\!\! \begin{array}{c} \includegraphics{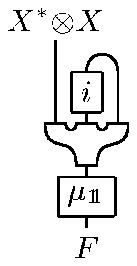} \end{array}
    = \!\!\! \begin{array}{c} \includegraphics{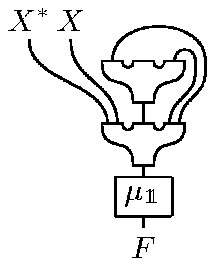} \end{array}
    \!\!\! = \!\!\! \begin{array}{c} \includegraphics{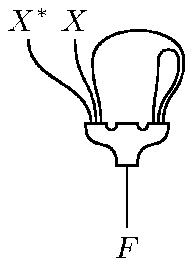} \end{array}
    \!\!\! = \!\!\! \begin{array}{c} \includegraphics{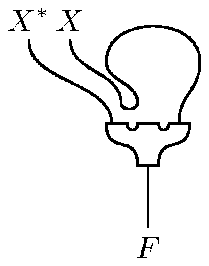} \end{array} \\[1.5ex]
    & = i_{\unitobj}(X).
  \end{align*}
  By the universality, $\mathbf{v}_F^{-1} \circ \psi = \id_F$. Hence $\psi = \mathbf{v}_F$.
\end{proof}

\subsection{Finiteness of the order of $E^{(n)}_{\mathbf{A}}$}

We show that the order of the map $E_{\mathbf{A}}^{(n)}$ is finite and give a bound of the order in terms of the distinguished invertible object of $\mathcal{C}$. A key observation for our proof is that there is an isomorphism
\begin{equation}
  \label{eq:End-otimes-n-iso-1}
  \Hom_{\mathcal{C}}(\unitobj, A^{\otimes n}) \cong \END(\otimes^n),
\end{equation}
where $\otimes^n: \mathcal{C}^n \to \mathcal{C}$ is the functor defined by
\begin{equation*}
  (V_1, \dotsc, V_n) \mapsto V_1 \otimes \dotsb \otimes V_n \quad (V_1, \dotsc, V_n \in \mathcal{C}).
\end{equation*}
As in the proof of Proposition~\ref{prop:adj-obj-Z}, it is easier to deal with $\mathbf{F} = (F, \psi)$ of that proposition than the adjoint object $\mathbf{A}$ itself. Thus, instead of establishing an isomorphism as in~\eqref{eq:End-otimes-n-iso-1}, we prove the following proposition:

\begin{proposition}
  \label{prop:End-otimes-n-iso}
  Define $\Phi(f) \in \END(\otimes^n)$ by
  \begin{equation*}
    \Phi(f)_{X_1, \dotsc, X_n} =
    \begin{array}{c}
      \includegraphics{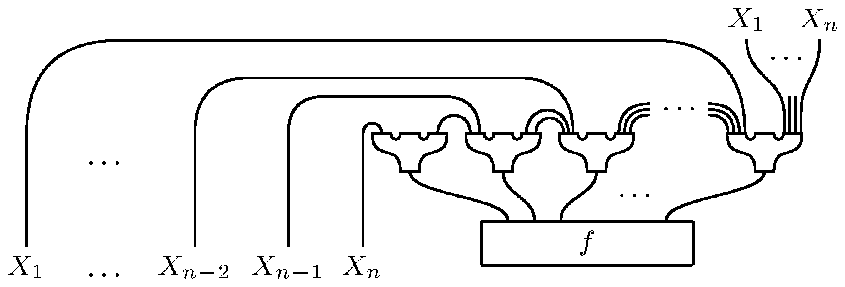}
    \end{array}
  \end{equation*}
  for a morphism $f: F^{\otimes n} \to \unitobj$ in $\mathcal{C}$ and objects $X_1, \dotsc, X_n \in \mathcal{C}$. Then
  \begin{equation}
    \label{eq:End-otimes-n-iso-2}
    \Phi: \Hom_{\mathcal{C}}(F^{\otimes n}, \unitobj) \to \END(\otimes^n),
    \quad f \mapsto \Phi(f)
  \end{equation}
  is an isomorphism of vector spaces.
\end{proposition}

An isomorphism $\Hom_{\mathcal{C}}(F^{\otimes n}, \unitobj) \cong \END(\otimes^n)$ is given in \cite[\S2.1]{MR2501845} under the assumption that $\mathcal{C}$ is a ribbon category. Replacing the braiding in the arguments in \cite[\S2.1]{MR2501845} with the half-braiding for $F$, one can see that \eqref{eq:End-otimes-n-iso-2} is an isomorphism without the assumption that $\mathcal{C}$ is a ribbon category. For the sake of completeness, we provide a detailed proof for the bijectivity of \eqref{eq:End-otimes-n-iso-2}.

\begin{proof}
  Define $\Phi': \Hom_{\mathcal{C}}(Z^n(\unitobj), \unitobj) \to \END(\otimes^n)$ by
  \begin{equation*}
    \Phi'(f)_{X_1, \dotsc, X_n} = (\id_{X_1 \otimes \dotsb \otimes X_n} \otimes f i) \circ 
    (\coev_{X_1 \otimes \dotsb \otimes X_n} \otimes \id_{X_1 \otimes \dotsb \otimes X_n})
  \end{equation*}
  for $f: Z^n(\unitobj) \to \unitobj$ and $X_1, \dotsc, X_n \in \mathcal{C}$, where $i = i^{(n)}_{\unitobj}(X_1, \dotsc, X_n)$ is the morphism given in \eqref{eq:universal-dinat-trans-n}. The dinaturality of $i$ in the variables $X_1, \dotsc, X_n \in \mathcal{C}$ implies the naturality of $\Phi'(f)$ in $X_1, \dotsc, X_n \in \mathcal{C}$, and the universality implies that the map $\Phi'$ is in fact an isomorphism of vector spaces; see \cite[\S5.3]{MR2869176} for details and more general result.

  Next, define an isomorphism $\widetilde{\kappa}_n: Z^n(\unitobj) \to F^{\otimes n}$ by $\widetilde{\kappa}_1 = \id_{F}$ and
  \begin{equation*}
    \widetilde{\kappa}_{m+1} = (\id_{F} \otimes \widetilde{\kappa}_m) \circ \kappa_{\unitobj, Z^{m}(\unitobj)}
  \end{equation*}
  for $m \ge 1$, where $\kappa_{V,M}: Z(V \otimes M) \to Z(V) \otimes M$ ($V \in \mathcal{C}$, $M \in \leftmod{Z}$) is the fusion operator for the Hopf monad $Z$ given by~\eqref{eq:l-Hopf-operator}. We note that $\kappa_{V, M}$ is graphically characterized as follows:
  \begin{equation*}
    \includegraphics{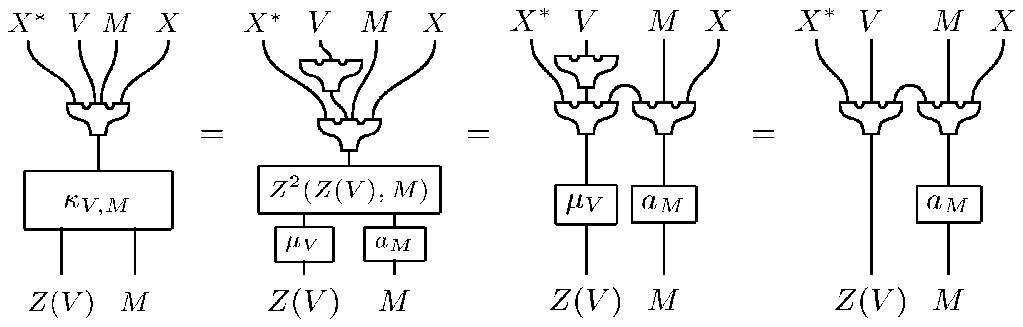}
  \end{equation*}
  By using this, we compute
  \begin{equation*}
    \widetilde{\kappa}_n \circ i_{\unitobj}^{(n)}(X_1, \dotsc, X_n):
    X_n^* \otimes \dotsb \otimes X_1^* \otimes X_1 \otimes \dotsb \otimes X_n \to F^{\otimes n}
  \end{equation*}
  as in Figure~\ref{fig:computation-kappa-n}. One now sees that the map $f \mapsto \Phi'(\widetilde{\kappa}_n \circ f)$ is in fact the map $\Phi$ in consideration. Hence $\Phi$ is an isomorphism.
\end{proof}

\begin{figure}
  \includegraphics{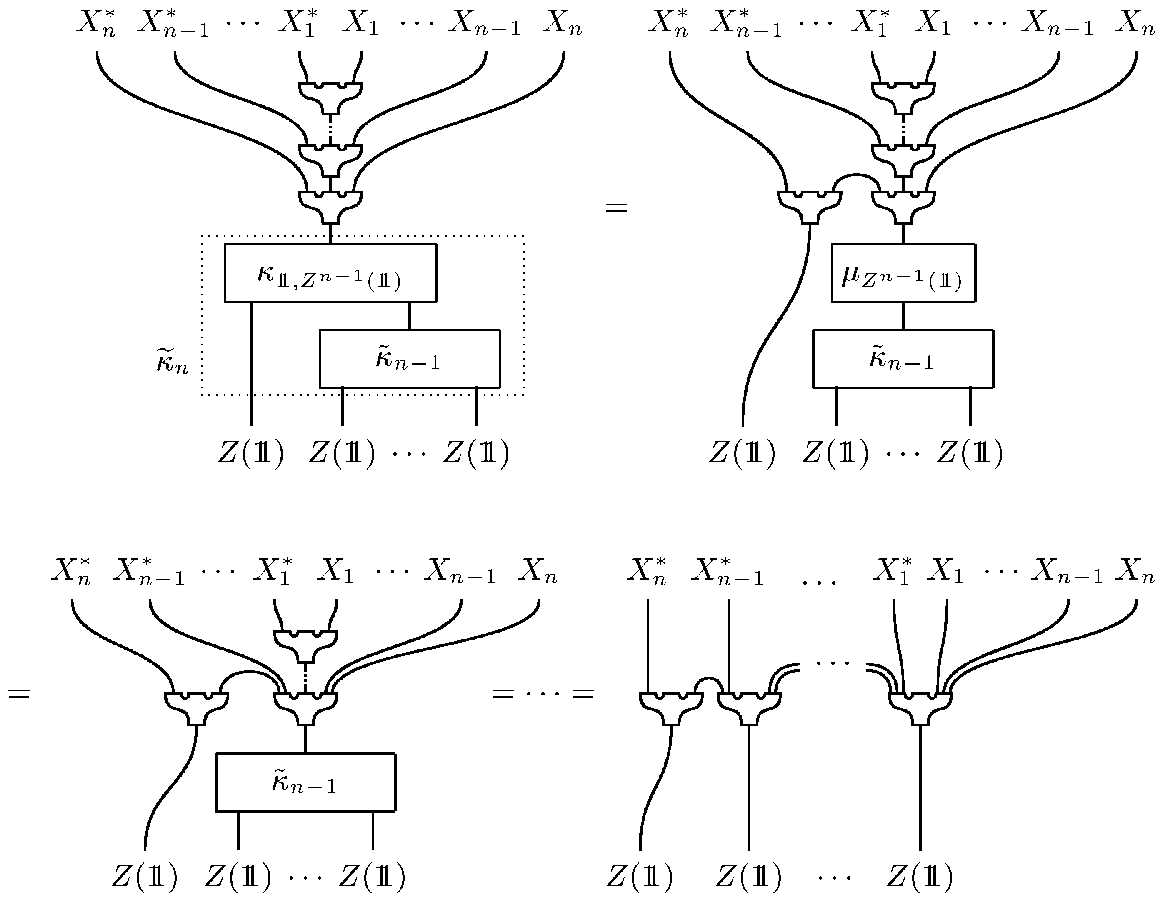}
  \caption{Computation of $\widetilde{\kappa}_n \circ i_{\unitobj}^{(n)}(X_1, \dotsc, X_n)$}
  \label{fig:computation-kappa-n}
\end{figure}

Let $E_{\otimes}^{(n)}: \END(\otimes^n) \to \END(\otimes^n)$ be the map induced by $\overline{E}_{\mathbf{F}}^{(n)}$ via the isomorphism $\Phi$ of Proposition \ref{prop:End-otimes-n-iso}. Since $\overline{E}_{\mathbf{F}}^{(n)}$ is conjugate to $E_{\mathbf{A}}^{(n)}$, so is $E_{\otimes}^{(n)}$. It seems to be difficult to express this map in a familiar way. However, the $n$-th power of $E_{\otimes}^{(n)}$ is an exception; it can be expressed in terms of the double duality functor $(-)^{**}$ as follows:

\begin{proposition}
  \label{prop:End-otimes-n-E-map}
  For all $\alpha \in \END(\otimes^n)$ and $X_1, \dotsc, X_n \in \mathcal{C}$, we have
  \begin{equation}
    \label{eq:End-otimes-n-E-map-1}
    (E_{\otimes}^{(n)})^n(\alpha)_{X_1^{**}, \dotsc, X_n^{**}} = (\alpha_{X_1, \dotsc, X_n})^{**}.
  \end{equation}
\end{proposition}
\begin{proof}
  Write $\overline{E} = \overline{E}_{\mathbf{F}}^{(n)}$ and $\mathbf{F} = (F, \psi)$. By Proposition~\ref{prop:End-otimes-n-iso}, we have
  \begin{equation*}
    \includegraphics{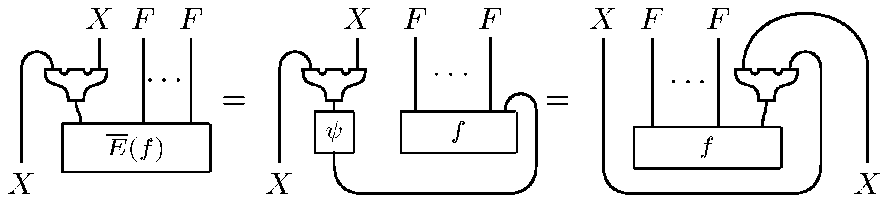}
  \end{equation*}
  for all objects $X \in \mathcal{C}$ and morphisms $f: F^{\otimes n} \to \unitobj$. By using this formula repeatedly, we have
  \begin{equation}
    \label{eq:End-otimes-n-E-map-2}
    \Phi(\overline{E}{}^n(f))_{X_1^{**}, \dotsc, X_n^{**}} = (\Phi(f)_{X_1, \dotsc, X_n})^{**}
  \end{equation}
  for all $X_1, \dotsc, X_n \in \mathcal{C}$ (see Figure~\ref{fig:End-otimes-n-E-map-proof} for the detailed computation for \eqref{eq:End-otimes-n-E-map-2} with $n = 4$). Hence~\eqref{eq:End-otimes-n-E-map-1} follows. 
\end{proof}

\begin{figure}
  \includegraphics{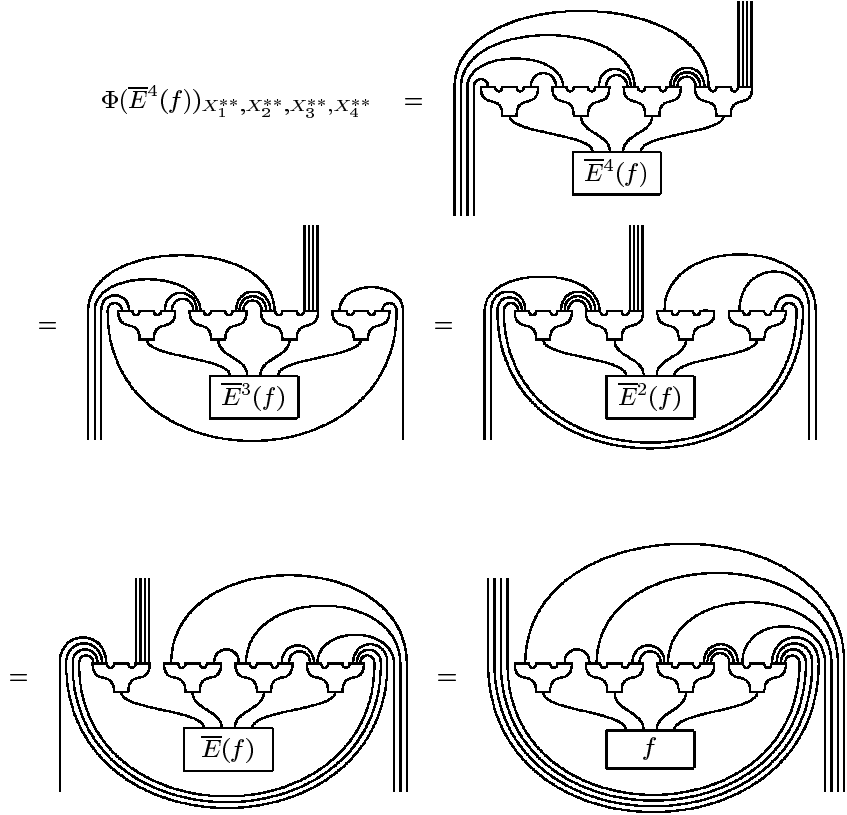}
  \caption{A detailed proof of~\eqref{eq:End-otimes-n-E-map-2} for $n = 4$}
  \label{fig:End-otimes-n-E-map-proof}
\end{figure}

The distinguished invertible object $D \in \mathcal{C}$ is an invertible object introduced by Etingof, Nikshych and Ostrik in \cite{MR2097289}. Generalizing Radford's $S^4$-formula for finite-dimensional Hopf algebras, they showed that there exists an isomorphism
\begin{equation}
  \label{eq:Radford-S4-for-FTCs}
  (-)^{****} \cong D \otimes (-) \otimes D^*
\end{equation}
of monoidal functors. In particular, some positive powers of $(-)^{**}$ are isomorphic to $\id_{\mathcal{C}}$ as monoidal functors.

Following, let $m(\mathcal{C})$ denote the smallest positive integer $m$ such that the $m$-th power of $(-)^{**}$ is isomorphic to $\id_{\mathcal{C}}$ as a monoidal functor. For example, $m(\mathcal{C}) = 1$ if and only if $\mathcal{C}$ admits a pivotal structure. By~\eqref{eq:Radford-S4-for-FTCs}, we have
\begin{equation}
  \label{eq:m(C)-bound}
  m(\mathcal{C}) \mid 2 \ord(D),
\end{equation}
where $\ord(D)$ is the smallest positive integer $r$ such that $D^{\otimes r} \cong \unitobj$.

\begin{theorem}
  \label{thm:En-order}
  The order of $E_{\otimes}^{(n)}$ divides $n \cdot m(\mathcal{C})$.
\end{theorem}

Since $E_{\mathbf{A}}^{(n)}$ and $\overline{E}_{\mathbf{F}}^{(n)}$ are conjugate to $E_{\otimes}^{(n)}$, the orders of $E_{\mathbf{A}}^{(n)}$ and $\overline{E}_{\mathbf{F}}^{(n)}$ are equal to the order of $E_{\otimes}^{(n)}$ and, in particular, divide $n \cdot m(\mathcal{C})$. By \eqref{eq:m(C)-bound}, the same order of these three maps divides $2 n \ord(D)$.

\begin{proof}
  Let $G$ be the $m$-th power of $(-)^{**}$, where $m = m(\mathcal{C})$. We note that $G$ is a strict monoidal functor since so is $(-)^{**}$ by our assumption. By definition, there exists an isomorphism $g: G \to \id_{\mathcal{C}}$ of monoidal functors. Now let $\alpha \in \END(\otimes^n)$ be a natural transformation and consider the following diagram:
  \begin{equation*}
    \begin{CD}
      G(X_1) \otimes \dotsb \otimes G(X_n) @>{g \otimes \dotsb \otimes g}>>
      X_1 \otimes \dotsb \otimes X_n @>{g^{-1}}>> G(X_1 \otimes \dotsb \otimes X_n) \\
      @V{\alpha_{G(X_1), \dotsc, G(X_n)}}V{}V
      @V{\alpha_{X_1, \dotsc, X_n}}V{}V
      @V{}V{G(\alpha_{X_1, \dotsc, X_n})}V \\
      G(X_1) \otimes \dotsb \otimes G(X_n) @>>{g \otimes \dotsb \otimes g}>
      X_1 \otimes \dotsb \otimes X_n @>>{g^{-1}}> G(X_1 \otimes \dotsb \otimes X_n), \\
    \end{CD}
  \end{equation*}
  where $X_1, \dotsc, X_n \in \mathcal{C}$. Since $\alpha$ and $g$ are natural, this diagram commutes. By the definition of monoidal natural transformations, the compositions along the top and the bottom row are the identities. By Proposition \ref{prop:End-otimes-n-E-map}, we have
  \begin{equation*}
    (E_{\otimes}^{(n)})^{n m}(\alpha)_{X_1, \dotsc, X_n}
    = G(\alpha_{G^{-1}(X_1), \dotsc, G^{-1}(X_m)}) = \alpha_{X_1, \dotsc, X_n}
  \end{equation*}
  for all $X_1, \dotsc, X_n \in \mathcal{C}$. Hence the order of $E_{\otimes}^{(n)}$ divides $n m$.
\end{proof}

Let $m'(\mathcal{C})$ be the smallest positive integer $m$ such that the $m$-th power of $(-)^{**}$ is isomorphic to $\id_{\mathcal{C}}$ as a functor. In the case where $n = 1$, the proof of the above theorem works even if $g$ is only an isomorphism of functors. Thus we obtain the following theorem:

\begin{theorem}
  \label{thm:En-order-1}
  The order of $E_{\otimes}^{(1)}$ divides $m'(\mathcal{C})$.
\end{theorem}

The following two corollaries are immediate consequences of the above results:

\begin{corollary}
  For every positive integer $n$, the sequence $\{ \nu_{n, r}(\mathbf{A}_{\mathcal{C}}) \}_{r = 0, 1, \dotsc}$ is periodic. The period of the sequence divides $n \cdot m(\mathcal{C})$.
\end{corollary}

For each integer $r \ge 1$, we define $\mathcal{O}_r$ to be the subring of $k$ generated by the roots of the polynomial $X^r - 1$. Since the trace of a linear operator is the sum of the eigenvalues, we have:

\begin{corollary}
  $\nu_{n, r}(\mathbf{A}_{\mathcal{C}}) \in \mathcal{O}_{n \cdot m(\mathcal{C})}$ for all $n \in \mathbb{Z}_{>0}$ and $r \in \mathbb{Z}$.
\end{corollary}

\section{Adjoint representations of Hopf algebras}
\label{sec:adj-rep-Hopf}

\subsection{Adjoint object in $\leftmod{H}$}
\label{subsec:adj-obj-Hopf-case}

Let $H$ be a finite-dimensional Hopf algebra over an algebraically closed field $k$ with comultiplication $\Delta$, counit $\epsilon$ and antipode $S$. In this section, we describe the adjoint object in $\mathcal{C} = \leftmod{H}$ and study its indicators in detail.

For this purpose, it is useful to realize the center of $\leftmod{H}$ as the category of Yetter-Drinfeld modules. Recall that a Yetter-Drinfeld module over $H$ is a pair $(M, \rho)$ consisting of a left $H$-module $M$ and a left $H$-coaction $\rho_M: M \to H \otimes_k M$ such that the Yetter-Drinfeld condition
\begin{equation*}
  \rho_M(h m) = h_{(1)} m_{(-1)} S(h_{(3)}) \otimes h_{(2)} m_{(0)}
\end{equation*}
holds for all $h \in H$ and $m \in M$, where $\rho_M(m) = m_{(-1)} \otimes m_{(0)}$. We denote by ${}_H^H \mathcal{YD}$ the category of finite-dimensional Yetter-Drinfeld modules over $H$. This category is a braided monoidal category \cite[\S10.6]{MR1243637} and is isomorphic to $\mathcal{Z}(\leftmod{H})$ as a braided monoidal category via
\begin{equation}
  \label{eq:center-and-YD-modules}
  {}^H_H \YD \to \mathcal{Z}(\leftmod{H}),
  \quad (M, \rho_M) \mapsto (M, \sigma_M),
\end{equation}
where $\sigma_M: M \otimes (-) \to (-) \otimes M$ is given by $\sigma_M(V)(m \otimes v) = m_{(-1)}v \otimes m_{(0)}$ for $v \in V \in \leftmod{H}$ and $m \in M$. Note that the isomorphism \eqref{eq:center-and-YD-modules} commutes with the forgetful functors to $\leftmod{H}$.

In what follows, we identify $\mathcal{Z}(\leftmod{H})$ with ${}^H_H \YD$ via \eqref{eq:center-and-YD-modules}. Then a right adjoint functor of the forgetful functor $U: \mathcal{Z}(\leftmod{H}) \to \leftmod{H}$ is given as follows: Given $V \in \leftmod{H}$, we make $I(V) = H \otimes_k V$ into a Yetter-Drinfeld module over $H$ by defining the action and the coaction of $H$ by
\begin{equation*}
  h \cdot (a \otimes v)
  = h_{(1)} a S(h_{(3)}) \otimes h_{(2)}v
  \text{\quad and \quad}
  a \otimes v \mapsto a_{(1)} \otimes (a_{(2)} \otimes v)
\end{equation*}
for $a, h \in H$ and $v \in V$, respectively. The assignment $V \mapsto I(V)$ naturally extends to a functor $I: \leftmod{H} \to {}^H_H \YD$, which can be considered as a variant of Radford's induction functor \cite[\S2]{MR2019635}. Now we define natural transformations $\eta$ and $\varepsilon$ by
\begin{equation*}
  \renewcommand{\arraystretch}{1.25}
  \begin{array}{ccc}
    \eta_M: M \to I U(M),
    & m \mapsto m_{(-1)} \otimes m_{(0)}
    & (m \in M \in \leftmod{H}), \\
    \varepsilon_V: U I(V) \to V,
    & a \otimes v \mapsto \epsilon(a) v
    & (v \in V \in \leftmod{H}, a \in H).
  \end{array}
\end{equation*}
It is easy to check that $\langle U, I, \eta, \varepsilon \rangle: {}^H_H \YD \rightharpoonup \leftmod{H}$ is an adjunction. By using this adjunction, we now give the following description of the adjoint object:

\begin{proposition}
  \label{prop:adj-obj-H-mod}
  Let $A = H_{\ad}$ be the adjoint representation of $H$, {\it i.e.}, the vector space $H$ with the left $H$-module structure $\triangleright$ given by
  \begin{equation*}
    h \triangleright a = h_{(1)} a S(h_{(2)}) \quad (h \in H, a \in A),
  \end{equation*}
  and let $\phi: A \to A^{**}$ be the isomorphism of $H$-modules defined by
  \begin{equation*}
    \langle \phi(a), p \rangle = \langle p, S^2(a) \rangle \quad (a \in A, p \in A^*).
  \end{equation*}
  Then $\mathbf{A}_{\leftmod{H}} = (A, \phi)$ is the adjoint object for $\leftmod{H}$.
\end{proposition}
\begin{proof}
  Define $\rho: A \to H \otimes_k A$ by $\rho = \Delta$. The canonical isomorphism $X \otimes_k k \cong X$ of vector spaces induces an isomorphism $I(k) \cong (A, \rho)$ of Yetter-Drinfeld modules over $H$. In view of Proposition~\ref{prop:adj-obj-Z}, it is sufficient to show that the map $\phi$ is equal to the Drinfeld isomorphism $\mathbf{u}_A$. To compute $\mathbf{u}_A$, we fix a basis $\{ a_i \}$ of $A$ and denote by $\{ a_i^* \}$ and $\{ a_i^{**} \}$ the bases of $A^*$ and $A^{**}$ dual to $\{ a_i \}$ and $\{ a_i^* \}$, respectively. Then, for all $a \in A$ and $p \in A^*$, we compute:
  \begin{align*}
    \langle \mathbf{u}_A(a), p \rangle
    & = \Big\langle (\eval_{A} \otimes \id) (\sigma_{A,A^*} \otimes \id) (\id \otimes \coev_{A^*}) (a), p \Big\rangle \\
    & = \sum_i \langle a_{(1)} a_i^*, a_{(2)} \rangle \langle a_i^{**}, p \rangle \\
    & = \sum_i \langle a_i^*, S(a_{(1)}) \triangleright a_{(2)} \rangle \langle p, a_i \rangle \\
    & = \sum_i \langle a_i^*, S^2(a) \rangle \langle p, a_i \rangle = \langle p, S^2(a) \rangle
    = \langle \phi(a), p \rangle.
    \qedhere
  \end{align*}
\end{proof}

Let $t \in \Hom_H(k, A^{\otimes n})$. Writing $t(1) = \sum_{j} t^{1}_j \otimes \dotsb \otimes t^n_j$, we compute:
\begin{align*}
  E_{\mathbf{A}}^{(n)}(t)(1)
  & = (\eval_A \otimes \id_A \otimes \dotsb \otimes \id_A) (\id_{A^*} \otimes t \otimes \phi^{-1}) \coev_{A^*} (1) \\
  & = \sum_{i, j} \langle a_i^*, t^1_{j} \rangle t^2_{j} \otimes \dotsb \otimes t^n_{j} \otimes \phi^{-1}(a_i^{**}) \\
  & = \sum_{j} t^2_{j} \otimes \dotsb \otimes t^n_{j} \otimes S^{-2}(t^1_j),
\end{align*}
where $\{ a_i^* \}$ and $\{ a_i^{**} \}$ have the same meaning as in the proof of the above proposition. 
Given $V \in \leftmod{H}$, we denote by $V^H$ the set of $H$-invariants in $V$. Under the canonical isomorphism $\Hom_H(k, V) \cong V^H$ with $V = A^{\otimes n}$, the map
\begin{equation*}
  E_{\ad}^{(n)}: (A^{\otimes n})^H \to (A^{\otimes n})^H,
  \quad \sum t^1_j \otimes \dotsb \otimes t^n_j
  \mapsto \sum t^2_{j} \otimes \dotsb \otimes t^n_{j} \otimes S^{-2}(t^1_j)
\end{equation*}
is conjugate to $E_{\mathbf{A}}^{(n)}$. Hence we obtain:

\begin{proposition}
  $\nu_{n,r}(\mathbf{A}_{\leftmod{H}}) = \Trace((E_{\ad}^{(n)})^r)$ for all $n \in \mathbb{Z}_{>0}$ and $r \in \mathbb{Z}$.
\end{proposition}

Therefore $\nu_{n,r}(\mathbf{A}_{\leftmod{H}})$ is computable in the following sense: Suppose that one knows the structural coefficients of $H$ with respect to a basis $\{ a_i \}$ of $H$. Then one can obtain a basis of $(A^{\otimes n})^H$ by solving the system of linear equations
\begin{equation*}
  \sum_{i_1, \dotsc, i_n = 1}^m (a_i - \epsilon(a_i)1_H) \triangleright (x_{i_1, \dotsc, i_n} a_{i_1} \otimes \dotsb \otimes a_{i_n}) = 0
  \quad (i = 0, \dotsc, m)
\end{equation*}
with variables $x_{i_1, \dotsc, i_n} \in k$, where $\triangleright$ is the action of $H$ on $A^{\otimes n}$. Since $E_{\ad}^{(n)}$ is defined only by using the antipode and the transposition of tensorands, one can represent it as a matrix with respect to the basis of $(A^{\otimes n})^H$. It is obvious that the trace of a matrix is computable.

Note that $A^H = {\rm Cent}(H)$, the center of $H$. Specializing the previous proposition to $n = 1$, we get:

\begin{proposition}
  $\nu_{1,r}(\mathbf{A}_{\leftmod{H}}) = \Trace(S^{-2r}|_{{\rm Cent}(H)})$.
\end{proposition}

Hence, if $S^{2m}$ is inner for some $m > 0$, then $\{ \nu_{1, r}(\mathbf{A}_{\leftmod{H}}) \}_{r = 0, 1, \dotsc}$ is a periodic sequence whose period divides $m$. This observation is a Hopf-algebraic counterpart of Theorem~\ref{thm:En-order-1}. Indeed, with the notation in that theorem, we have
\begin{equation*}
  m'(\leftmod{H}) = \min \{ m \in \mathbb{Z}_{>0} \mid \text{$S^{2m}$ is inner} \}.
\end{equation*}

\subsection{Semisimple case}

We now consider the case where $H$ is a finite-dimensional semisimple Hopf algebra over an algebraically closed field $k$ of characteristic zero. Then there uniquely exists an integral $\Lambda \in H$ such that $\epsilon(\Lambda) = 1$ (the {\em normalized integral}). By using this element, we define a linear form $\chi_{\ad}$ on $H$ by
\begin{equation*}
  \langle \chi_{\ad}, h \rangle = \langle \lambda, h_{(1)} \Lambda_{(1)} S(h_{(2)}) S(\Lambda_{(2)}) \rangle
\end{equation*}
for $h \in H$, where $\lambda \in H^*$ is the integral on $H$ such that $\langle \lambda, \Lambda \rangle = 1$.

\begin{theorem}
  Let $n \in \mathbb{Z}_{> 0}$ and $r \in \mathbb{Z}$. With the above notations, the $(n, r)$-th FS indicator of the adjoint object in $\leftmod{H}$ is given by
  \begin{equation*}
    \nu_{n,r}(\mathbf{A}_{\leftmod{H}})
    = \prod_{c = 0}^{d - 1}
    \langle \chi_{\ad}, \Lambda_{(i_{c, 0})} \Lambda_{(i_{c,1})} \dotsb \Lambda_{(i_{c, n/d - 1})} \rangle,
  \end{equation*}
  where $d = \gcd(n, r)$ and $i_{a, b} \in \{1, \cdots, n\} \ (0 \le a < d$, $0 \le b < n/d)$ is a unique integer such that $i_{a, b} \equiv a + b r + 1 \pmod{n}$.
\end{theorem}

For example,
\begin{equation*}
  \nu_{10,4}(\mathbf{A}_{\leftmod{H}})
  =     \langle \chi_{\ad}, \Lambda_{(1)} \Lambda_{(5)} \Lambda_{(9)}  \Lambda_{(3)} \Lambda_{(7)} \rangle
  \cdot \langle \chi_{\ad}, \Lambda_{(2)} \Lambda_{(6)} \Lambda_{(10)} \Lambda_{(4)} \Lambda_{(8)} \rangle.
\end{equation*}

\begin{proof}
  Let $\nu_{n,r}^\mathrm{NS}$ denote the $(n, r)$-th FS indicator of Ng and Schauenburg with respect to the canonical pivotal structure of $\leftmod{H}$. By Theorem \ref{thm:adj-obj-ind-NS}, we have
  \begin{equation}
    \label{eq:ind-adj-ss-1}
    \nu_{n,r}(\mathbf{A}_{\leftmod{H}}) = \nu_{n,r}^\mathrm{NS}(H_{\ad}).
  \end{equation}
  By using the integrals, $\nu_{n,r}^\mathrm{NS}(V)$ for $V \in \leftmod{H}$ is expressed as
  \begin{equation}
    \label{eq:ind-adj-ss-2}
    \nu_{n,r}^\mathrm{NS}(V) = \prod_{c = 0}^{d - 1}
    \langle \chi_V, \Lambda_{(i_{c, 0})} \Lambda_{(i_{c,1})} \dotsb \Lambda_{(i_{c, n/d - 1})} \rangle,
  \end{equation}
  where $\chi_V$ is the character of $V$ (see \cite[\S2.3]{MR2213320} for the case of $d = 1$; the general case is proved in a similar way). 
  Using Radford's trace formula \cite[Theorem 2]{MR1265853} and some identities on integrals \cite[Proposition 4]{MR1265853}, we have
  \begin{equation}
    \label{eq:ind-adj-ss-3}
    \langle \chi_{A}, h \rangle
    = \langle \lambda, S(\Lambda_{(2)}) (h \triangleright \Lambda_{(1)}) \rangle
    = \langle \chi_{\ad}, h \rangle
  \end{equation}
  for $h \in H$. Now the result follows from~\eqref{eq:ind-adj-ss-1}, \eqref{eq:ind-adj-ss-2} and \eqref{eq:ind-adj-ss-3}.
\end{proof}

\begin{corollary}
  Let $G$ be a finite group. For all $n \in \mathbb{Z}_{>0}$ and $r \in \mathbb{Z}$, we have
  \begin{equation*}
    \nu_{n, r}(\mathbf{A}_{\leftmod{kG}}) = \frac{1}{|G|} \sum_{g \in G} |C_G(g^{n/d})|^d,
  \end{equation*}
  where $d = \gcd(n, r)$ and $C_G(x) = \{ y \in G \mid x y = y x \}$.
\end{corollary}
\begin{proof}
  Indeed, the normalized integral is given by $\Lambda = |G|^{-1} \sum_{g \in G} g$, and the linear form $\chi_{\ad}$ is given by $\chi_{\ad}(x) = |C_G(x)|$ for $x \in G$.
\end{proof}

\subsection{Non-semisimple examples}

For non-semisimple Hopf algebras, we do not know any general formula to compute $\nu_{n,r}(\mathbf{A}_{\leftmod{H}})$. Thus it is an interesting problem to know its value. Here we determine $\nu_{1,r}(\mathbf{A}_{\leftmod{H}})$ for certain non-semisimple and non-pivotal Hopf algebras in a direct way.

Suppose that $k$ is an algebraically closed field of characteristic zero. Let $N$ and $m$ be integers such that $0 < m < N$ and $\gcd(N, m) = 1$, and let $\omega \in k$ be a primitive $N$-th root of unity. The algebra $H(\omega, m)$ is generated by $x$, $y$ and $g$ subject to the following relations:
\begin{equation*}
  g^N = 1,
  \quad x^N = y^N = 0,
  \quad g x = \omega x g,
  \quad g y = \omega^{-m} y g,
  \quad \text{and} \quad y x = \omega^m x y.
\end{equation*}
Note that the following set is a basis of this algebra:
\begin{equation*}
  \{ x^p y^q g^r \mid p, q, r = 0, \dotsc, N - 1 \}.
\end{equation*}
The algebra $H(\omega, m)$ has the Hopf algebra structure determined by
\begin{gather*}
  \Delta(x) = x \otimes 1 + g \otimes x,
  \quad \Delta(y) = y \otimes 1 + g^m \otimes y,
  \quad \Delta(g) = g \otimes g, \\
  \epsilon(x) = 0, \quad \epsilon(y) = 0, \quad \epsilon(g) = 1, \\
  S(x) = -g^{-1}x, \quad S(y) = -g^{-m} y, \quad S(g) = g^{-1}.
\end{gather*}
This Hopf algebra is an example of liftings of quantum planes and isomorphic to the Hopf algebra $\boldsymbol{\mathsf{h}}(\omega^{-1}, -m)$ of \cite{MR1659895}. Our result is:

\begin{theorem}
  \label{thm:1-r-th-ind-adj-book-Hopf}
  {\rm (a)} For all $r \in \mathbb{Z}$,
  \begin{equation*}
    \nu_{1, r}(\mathbf{A}_{\leftmod{H(\omega, 1)}}) = N (N - 1).
  \end{equation*}
  {\rm (b)} Suppose $m \ne 1$. Then, for all $r \in \mathbb{Z}$, we have
  \begin{equation*}
    \nu_{1, r}(\mathbf{A}_{\leftmod{H(\omega, m)}}) =
    \begin{cases}
      N & \text{if $r (m - 1) \equiv 0 \pmod{N}$}, \\
      0 & \text{otherwise}.
    \end{cases}
  \end{equation*}
\end{theorem}

Let $C(\omega, m)$ denote the center of the algebra $H(\omega, m)$. To prove this theorem, we need the following description of $C(\omega, m)$:

\begin{lemma}
  {\rm (a)} $C(\omega, 1)$ is spanned by
  \begin{equation}
    \label{eq:book-Hopf-center-basis-m=1}
    \{ x^{i} y^{i} g^{-i} \mid i = 0, \dotsc, N - 2 \}
    \cup \{ x^{N - 1} y^{N - 1} g^{j} \mid j = 0, \dotsc, N - 1 \}.
  \end{equation}
  {\rm (b)} If $m \ne 1$, then $C(\omega, m)$ is spanned by $\{ x^{i} y^{\overline{m}(i)} g^{-i} \mid i = 0, \dotsc, N - 1 \}$, where $\overline{m}(i)$ is the unique integer such that
  \begin{equation*}
    0 \le \overline{m}(i) < N \text{\quad and \quad} m \cdot \overline{m}(i) \equiv i \pmod{N}.
  \end{equation*}
\end{lemma}
\begin{proof}
  Note that an element $z \in H(\omega, m)$ is central if and only if
  \begin{equation}
    \label{eq:book-Hopf-central}
    g \triangleright z = z, \quad x \triangleright z = 0, \text{\quad and \quad} y \triangleright z = 0.
  \end{equation}
  One easily calculates that the adjoint action is given by
  \begin{align}
    \label{eq:book-Hopf-adj-x}
    x \triangleright x^p y^q g^r & = (1 - \omega^{p+r}) x^{p+1} y^q g^r, \\
    \label{eq:book-Hopf-adj-y}
    y \triangleright x^p y^q g^r & = \omega^{m p} (1 - \omega^{-m^2 q - m r}) x^p y^q g^r, \\
    \label{eq:book-Hopf-adj-g}
    g \triangleright x^p y^q g^r & = \omega^{p - m q} x^p y^q g^r
  \end{align}
  for $p, q, r \in \{ 0, \dotsc, N - 1 \}$. Now suppose that $z \in H(\omega, m)$ is central. By \eqref{eq:book-Hopf-adj-g} and the first equation of~\eqref{eq:book-Hopf-central}, $z$ must be of the form
  \begin{equation*}
    z = \sum_{i, j = 0}^{N - 1} c_{i j} x^i y^{\overline{m}(i)} g^{j}
  \end{equation*}
  for some $c_{i j} \in k$. By \eqref{eq:book-Hopf-central}, \eqref{eq:book-Hopf-adj-x} and~\eqref{eq:book-Hopf-adj-y}, we have
  \begin{equation*}
    0 = x \triangleright z = \sum_{i, j = 0}^{N - 1} c_{i j} (1 - \omega^{i+j}) x^{i+1} y^{\overline{m}(i)} g^{j}.
  \end{equation*}
  This implies that $c_{i j}$ is zero if $i + 1 < N$ and $i + j \not \equiv 0 \pmod{N}$. Hence,
  \begin{equation}
    \label{eq:book-Hopf-adj-x-3}
    z = \sum_{i = 0}^{N - 2} c'_{i} x^{i} y^{\overline{m}(i)} g^{-i}
    + \sum_{j = 0}^{N - 1} c''_{j} x^{N - 1} y^{\overline{m}(N - 1)} g^j,
  \end{equation}
  where $c'_i = c_{i, N - i}$ and $c''_j = c_{N - 1, j}$.

  {\rm (a)} If $m = 1$, then $\overline{m}(i) = i$ for all $i = 0, \dotsc, N - 1$. Therefore an element $z$ of the form~\eqref{eq:book-Hopf-adj-x-3} belongs to the space spanned by~\eqref{eq:book-Hopf-center-basis-m=1}. It is easy to see that the space spanned by~\eqref{eq:book-Hopf-center-basis-m=1} is contained by $C(\omega, 1)$. Hence the result follows.

  {\rm (b)} Now suppose that $m \ne 1$. The difference from the case where $m = 1$ is that an element $z$ of the form \eqref{eq:book-Hopf-adj-x-3} does not fulfill $y \triangleright z = 0$ in general. By~\eqref{eq:book-Hopf-adj-y},
  \begin{align*}
    y \triangleright z
    & = \sum_{j = 0}^{N - 1} c''_{j} \omega^{m (N - 1)}
    (1 - \omega^{-m^2 \cdot \overline{m}(N - 1)} - m j) x^{N - 1} y^{\overline{m}(N - 1) + 1} g^{j} \\
    & = \sum_{j = 0}^{N - 1} c''_{j}
    \omega^{- m} (1 - \omega^{- m (j - 1)}) x^{N - 1} y^{\overline{m}(N - 1) + 1} g^{j}.
  \end{align*}
  One can check that $\overline{m}(N - 1) \ne N - 1$ (otherwise $m \equiv 1 \pmod{N}$ would follow). Hence $y \triangleright z = 0$ implies $c''_{j} = 0$ for all $j = 0, \dotsc, N - 1$ except for $j = 1$. Summarizing the results so far, we have
  \begin{equation*}
    z \in C'(\omega, m) := {\rm span} \{ x^{i} y^{\overline{m}(i)} g^{-i} \mid i = 0, \dotsc, N - 1 \}.
  \end{equation*}
  Therefore $C(\omega, m) \subset C'(\omega, m)$. It is easy to show $C'(\omega, m) \subset C(\omega, m)$.
\end{proof}

\begin{proof}[Proof of Theorem~\ref{thm:1-r-th-ind-adj-book-Hopf}]
  {\rm (a)} If $m = 1$, then $S^2$ is the identity on $C(\omega, 1)$ (more precisely, $S^2$ is an inner automorphism implemented by $g$). By Proposition~\ref{prop:adj-obj-H-mod} and the previous lemma, we have
  \begin{equation*}
    \nu_{1,r}(\mathbf{A}_{\leftmod{H(\omega,1)}}) = \dim_k C(\omega, 1) = N (N - 1).
  \end{equation*}
  {\rm (b)} If $m \ne 1$, then $z_i := x^i y^{\overline{m}(i)} g^{-i}$ ($i = 0, \dotsc, N - 1)$ is a basis of $C(\omega, m)$ by the previous lemma. The map $E_{\ad}^{(1)}$ of Proposition~\ref{prop:adj-obj-H-mod} is represented by a diagonal matrix with respect to the basis $\{ z_i \}$, as follows:
  \begin{equation*}
    E_{\ad}^{(1)}(z_i) = S^{-2}(z_i) = \omega^{- r i (m - 1)} z_i
    \quad (i = 0, \dotsc, N - 1).
  \end{equation*}
  Hence $\nu_{1,r}(\mathbf{A}_{\leftmod{H(\omega,m)}}) = \Trace(E_{\ad}^{(1)})$ is given as stated.
\end{proof}

Two finite-dimensional Hopf algebras $A$ and $B$ are said to be {\em gauge equivalent} if $\leftmod{A}$ and $\leftmod{B}$ are equivalent as $k$-linear monoidal categories. To conclude this paper, we discuss the possibility of applications of our results to the problem on gauge equivalence of Hopf algebras. Namely, we investigate whether the following assertions are equivalent or not:
\begin{itemize}
\item [(1)] $H(\omega, m)$ and $H(\omega', m')$ are isomorphic as Hopf algebras.
\item [(2)] $H(\omega, m)$ and $H(\omega', m')$ are gauge equivalent.
\end{itemize}
It is obvious that (1) implies (2). However, it does not seem to be known whether the converse holds. In what follows, we write
\begin{equation*}
  \nu_{n,r}(\omega,m) = \nu_{n,r}(\mathbf{A}_{\leftmod{H(\omega,m)}})
\end{equation*}
for simplicity. By Theorems \ref{thm:adj-obj-invariance} and~\ref{thm:1-r-th-ind-adj-book-Hopf}, we can give very partial answers to the above question. For example, let $\omega_{27}$ be a primitive $27$-th root of unity. Then $H(\omega_{27}, 13)$ and $H(\omega_{27}, 14)$ are not gauge equivalent, since
\begin{equation*}
  \nu_{1,9}(\omega_{27}, 13) = 27 \quad \ne \quad \nu_{1,9}(\omega_{27}, 14) = 0.
\end{equation*}
The invariant $\nu_{1,r}(\omega, m)$ is not sufficient to solve the problem. To see this, we use the gauge invariant $\nu_{-1}^{\rm KMN}$, which we have mentioned in \S\ref{subsec:FS-ind-examples}. Since the element
\begin{equation*}
  \Lambda = \sum_{i = 0}^{N - 1} g^{i} x^{N - 1} y^{N - 1} \in H(\omega, m)
\end{equation*}
is a non-zero left integral, we have
\begin{equation*}
  \nu_{-1}^{\rm KMN}(H(\omega, m)) = \omega^{1 - m^2}
\end{equation*}
by~\eqref{eq:KMN-ind-(-1)}. For simplicity, we define $e(\omega, m)$ to be a unique integer such that
\begin{equation*}
  0 \le e(\omega, m) < \ord(\omega) \text{\quad and \quad} \omega^{e(\omega, m)} = \nu_{-1}^{\rm KMN}(H(\omega, m)).
\end{equation*}
Table~\ref{tab:book-Hopf-gauge-inv} displays the values of $\nu_{1,r}(\omega_{27}, m)$ and $e(\omega_{27}, m)$ for $r = 1, 3, 9$ and all possible $m$. One finds in the table many pairs such that $H(\omega, m)$ and $H(\omega, m')$ are not gauge equivalent but $\nu_{1,r}(\omega, m) = \nu_{1,r}(\omega, m')$ for all $r$.

\begin{table}
  \centering
  \renewcommand{\arraystretch}{1.15}
  \begin{tabular}{c|ccccccccccccc}
    $m$            &  1 &  2 &  4 &  5 &  7 &  8 & 10 & 11 & 13 \\ \hline
    $\nu_{1,1}$ & 702 &  0 &  0 &  0 &  0 &  0 &  0 &  0 &  0 \\
    $\nu_{1,3}$ & 702 &  0 &  0 &  0 &  0 &  0 & 27 &  0 &  0 \\
    $\nu_{1,9}$ & 702 &  0 & 27 &  0 & 27 &  0 & 27 &  0 & 27 \\
    $e$  &  0 &  3 & 15 & 24 & 21 & 10 & 19 & 13 &  7 \\
    \multicolumn{3}{c}{} \\
    $m$            & 14 & 16 & 17 & 19 & 20 & 22 & 23 & 25 & 26 \\ \hline
    $\nu_{1,1}$ & 0 &  0 &  0 &  0 &  0 &  0 &  0 &  0 &  0 \\
    $\nu_{1,3}$ & 0 &  0 &  0 & 27 &  0 &  0 &  0 &  0 &  0 \\
    $\nu_{1,9}$ & 0 & 27 &  0 & 27 &  0 & 27 &  0 & 27 &  0 \\
    $e$ &  7 & 13 & 19 & 10 & 21 & 24 & 15 &  3 & 0
  \end{tabular}
  \bigskip
  \caption{The values of $\nu_{1,r} = \nu_{1,r}(\omega_{27}, m)$ and $e = e(\omega_{27}, m)$}
  \label{tab:book-Hopf-gauge-inv}
\end{table}

The table also shows that $H(\omega_{27}, m)$ and $H(\omega_{27}, m')$ are gauge equivalent if and only if they are isomorphic as Hopf algebras. It is fortunate that we can obtain such a result; for other $\omega$ of different orders, even the combination of $\nu_{1,r}(\omega,m)$ and $e(\omega,m)$ does not seem to be enough to conclude the problem on the gauge equivalence of $H(\omega, m)$'s.

\def\cprime{$'$}

\end{document}